\documentclass{amsart}
\usepackage[utf8]{inputenc}
\usepackage{amsmath}
\usepackage{amsfonts}
\usepackage{amssymb}
\usepackage{amsthm}
\usepackage{amscd}
\usepackage{float}
\usepackage{tikz}
\usepackage{graphicx}
\usepackage[colorlinks=true]{hyperref}
\hypersetup{urlcolor=blue, citecolor=red}
\usepackage{hyperref}

\newtheorem{theorem}{Theorem}[section]
\newtheorem{corollary}[theorem]{Corollary}

\newtheorem{proposition}[theorem]{Proposition}

\theoremstyle{definition}
\newtheorem{definition}[theorem]{Definition}
\newtheorem{remark}[theorem]{Remark}

\newcommand{\R}{\mathbb{R}}
\newcommand{\N}{\mathbb{N}}
\newcommand{\C}{\mathbb{C}}
\newcommand{\T}{\mathbb{T}}
\newcommand{\Z}{\mathbb{Z}}
\newcommand{\Q}{\mathbb{Q}}

\def\XXint#1#2#3{{\setbox0=\hbox{$#1{#2#3}{\int}$ }
\vcenter{\hbox{$#2#3$ }}\kern-.6\wd0}}

\makeatletter
\@namedef{subjclassname@2020}{%
  \textup{2020} Mathematics Subject Classification}
\makeatother

\begin{document}

\title{Strichartz estimates for quasi-periodic functions and applications}

\author{Robert Schippa}
\address{UC Berkeley, Department of Mathematics, Evans Hall
Berkeley, CA 94720-3840}
\email{robert.schippa@gmail.com}

%\date{\today}
%\subjclass[2020]{35Q53, 42B37.}

\begin{abstract}
We show Strichartz estimates for quasi-periodic functions with decaying Fourier coefficients via $\ell^2$-decoupling. When we additionally average in time, further improvements can be obtained. Next, we apply multilinear refinements to show low regularity local well-posedness for nonlinear Schr\"odinger equations. For the cubic nonlinear Schr\"odinger equation the approach yields the sharp local well-posedness result.
\end{abstract}

\maketitle

\section{Introduction}

The local well-posedness of nonlinear dispersive equations with almost periodic initial data has recently received increased attention.
Two prominent models are the cubic nonlinear Schr\"odinger equation
\begin{equation}
\label{eq:CubicNLSIntro}
\left\{ \begin{array}{cl}
i \partial_t u + \partial_{x}^2 u &= \pm |u|^2 u, \quad (t,x) \in \R \times \R, \\
u(0) &= u_0
\end{array} \right.
\end{equation}
and the Korteweg-de Vries equation:
\begin{equation}
\label{eq:KdVIntro}
\left\{ \begin{array}{cl}
\partial_t u + \partial_x^3 u &= u \partial_x u , \quad (t,x) \in \R \times \R, \\
u(0) &= u_0.
\end{array} \right.
\end{equation}

Recall that a Bohr-almost periodic function \cite{Bohr1925} is obtained as a uniform limit of trigonometric polynomials. The set of almost periodic functions on $\R$ is denoted by $AP(\R)$. For these functions the Fourier series expansion holds in the $L^\infty$-norm:
\begin{equation}
\label{eq:FourierExpansionIntro}
f(x) = \sum_{\lambda \in \Lambda} e^{i \lambda x} a_\lambda, \quad \lambda \in \Lambda \subseteq \R: \text{ at most countably infinite}.
\end{equation}

T. Oh \cite{Oh2015} proved local well-posedness of \eqref{eq:CubicNLSIntro} in the Banach algebra of functions with absolutely summable Fourier coefficients. By local well-posedness we refer to existence, uniqueness, and continuous dependence of the solutions on the initial data. The result in \cite{Oh2015} was recently generalized by Papenburg \cite{Papenburg2024} who considered more general Banach algebras. Furthermore, Damanik \emph{et al.} \cite{DamanikLiXu2024A,DamanikLiXu2024B} considered power-type nonlinear Schr\"odinger equations with a nonlinearity at least quintic and the derivative nonlinear Schr\"odinger equation.

\smallskip

%Papenburg moreover considered general dispersion relations with derivative nonlinearity and showed local well-posedness for data with exponentially decaying Fourier coefficients. 
 Tsugawa \cite{Tsugawa2012} applied the Fourier restriction norm method due to Bourgain \cite{Bourgain1993A,Bourgain1993B} to show local well-posedness of \eqref{eq:KdVIntro}. He showed that his result is sharp in general in terms of Fourier decay.
Presently, we likewise aim to minimize the decay of the Fourier coefficients, which is required to prove local well-posedness. This can be viewed as a low regularity well-posedness theory for nonlinear dispersive equations with quasi-periodic initial data. The presented arguments generalize the theory in the periodic case, which was initiated in \cite{Bourgain1993A,Bourgain1993B}.

\smallskip

On the other hand, Chapouto--Killip--Vi\c{s}an \cite{ChapoutoKillipVisan2022} pointed out that \eqref{eq:KdVIntro} is ill-posed in $AP(\R)$ due to instant loss of continuity. This disproved the Deift conjecture \cite{Deift2017} stating that almost-periodic initial data to \eqref{eq:KdVIntro} yield solutions almost-periodic in time. Notably, an affirmative answer was given under more restrictive assumptions on the initial data by Damanik \emph{et al.}  \cite{DamanikGoldstein2016,BinderDamanikGoldsteinLukic2018,DamanikGoldstein2014}.
 A possible interpretation to \cite{ChapoutoKillipVisan2022} is that continuity is too strong an assumption for initial data to (nonlinear) dispersive equations. Instead we consider variants of the following norm:
\begin{equation*}
\| f \|_{\mathcal{L}_x^2}^2 = \lim_{L \to \infty} \frac{1}{2L} \int_{[-L,L]} |f(x)|^2 dx = \mathcal{M}(|f|^2).
\end{equation*}
The set of functions $f(x) = \sum_{\lambda \in \Lambda} a_\lambda e^{i \lambda x}$ with absolutely summable Fourier coefficients $(a_\lambda)_{\lambda \in \Lambda} \in \ell^1 $ is denoted with $\mathcal{A}_\omega$. For these functions we have
\begin{equation*}
\| f \|_{\mathcal{L}_x^2} = \big( \sum_{\lambda \in \Lambda} |a_\lambda|^2 \big)^{\frac{1}{2}}.
\end{equation*}
For $f \in \mathcal{L}^2_x(\R)$ we let $\sigma(f) = \{ \lambda \in \R^d : \mathcal{M}(f e^{-i \lambda \cdot})\neq 0 \}$. We define the space of Besicovitch-almost periodic functions $\mathcal{B}_2(\R)$ as closure of trigonometric polynomials under the $\mathcal{L}^2_x$-norm. \eqref{eq:FourierExpansionIntro} holds in $\mathcal{L}^2_x(\R)$.
%\begin{equation*}
%\mathcal{B}_2(\R) = \overline{\{ f(x) = \sum_{\lambda \in \Lambda_f} a_\lambda e^{i \lambda x}: \Lambda_f \text{ finite } \}}^{\| \cdot \|_{\mathcal{L}_x^2}}.
%\end{equation*}
 The fact that solutions to \eqref{eq:CubicNLSIntro} and \eqref{eq:KdVIntro} for real-valued initial data satisfy the formal conservation law
\begin{equation*}
\| u(t) \|_{\mathcal{L}_x^2} = \| u_0 \|_{\mathcal{L}_x^2}
\end{equation*}
indicates that an $\mathcal{L}^2_x$-based space is more appropriate to solve the quasiperiodic Cauchy problem.

\smallskip

For decaying initial data on the real line or periodic initial data, dispersive effects are the bedrock to establish a low regularity well-posedness theory.
A manifestation of dispersive effects are Fourier extension inequalities, which are also known as Strichartz estimates.

For linear Schr\"odinger equations on the real line these estimates (see \cite{KeelTao1998} and references therein) read
\begin{equation*}
\| e^{it \partial_x^2} u_0 \|_{L^p_t(\R;L^q_x(\R))} \lesssim \| u_0 \|_{L^2_x(\R)}, \quad p,q \geq 2, \quad \frac{2}{p} + \frac{1}{q} = \frac{1}{2}.
\end{equation*}
With this at hand, it is straight-forward to solve \eqref{eq:CubicNLSIntro} via the contraction mapping principle and obtain solutions with analytic dependence on the initial data in $L^2(\R)$. Since the mass is conserved, the local well-posedness theory is readily iterated to prove global well-posedness.

\smallskip

Due to weaker dispersive effects on the circle, less is known about Strichartz estimates for periodic functions. It still holds the $L^4_{t,x}$-Strichartz estimate without derivative loss (cf. \cite{Bourgain1993A,Zygmund1974}):
\begin{equation}
\label{eq:PeriodicL4StrichartzIntroduction}
\| e^{it \partial_x^2} u_0 \|_{L_t^4([0,1],L_x^4(\T))} \lesssim \| u_0 \|_{L^2(\T)}.
\end{equation}
%On the other hand, it is known that the endpoint estimate fails (cf. \cite{Bourgain1993A}):
%\begin{equation*}
%\| e^{it \partial_x^2} P_N u_0 \|_{L_t^6([0,1],L_x^6(\T))} \gtrsim \log(N)^{\frac{1}{6}} \| P_N u_0 \|_{L^2}.
%\end{equation*}
%Above $P_N$ denotes the frequency projection to frequencies of size comparable to $N$.
%For recent progress on the constant we refer to \cite{GuthMaldagueWang2022,GuoLiYung2024}.

\medskip

Here we point out how the C\'ordoba--Fefferman square function estimate \cite{Cordoba1979,Cordoba1982,Fefferman1973} and decoupling arguments due to Bourgain--Demeter \cite{BourgainDemeter2015} yield Strichartz estimates for linear solutions with quasi-periodic initial data. Here it becomes important to approximate an exponential sum with an oscillatory integral. The latter is amenable to decoupling. This approach of estimating exponential sums goes back to Bourgain \cite{Bourgain2013}. We also remark that substitutes of Strichartz estimates in $L^p$-based spaces for $p>2$, again relating to initial data with less decay compared to $L^2$-based spaces, were previously investigated in \cite{Schippa2022SlowDecay,RozendaalSchippa2023}.
 
\smallskip 
 
Let $f \in \mathcal{B}_2(\R)$ be a quasi-periodic function with $\sigma(f) \subseteq \Lambda $. Here $\nu \in \N$, and suppose that $\omega \in \R^{\nu}_{>0}$ is \emph{non-resonant}, i.e., $\{ \omega_1,\ldots,\omega_{\nu} \}$ are linearly independent over $\Q$. Let $\Lambda = \omega_1 \Z + \ldots + \omega_{\nu} \Z$ and denote by $\langle n \rangle_\omega = n \cdot \omega$ for $n \in \Z^{\nu}$. We have the representation in $\mathcal{L}^2_x$:
\begin{equation*}
f(x) = \sum_{n \in \Z^{\nu}} e^{i \langle n \rangle_\omega \cdot x} \hat{f}(\langle n \rangle_\omega).
\end{equation*}

\smallskip

We quantify the Fourier decay in terms of the height $|n|$ of the Fourier coefficients, which leads us to the following definition of Sobolev-type spaces:
\begin{equation*}
\| f \|^2_{\mathcal{H}^s_{\Lambda}} = \sum_{n \in \Z^{\nu}} \langle n \rangle^{2s} |\hat{f}(\langle n \rangle_\omega)|^2.
\end{equation*}

\medskip

For $C,N \in 2^{\N_0}$ let $R_C$ denote the projection to frequencies of height $C$ and $P_N$ denote the projection to frequencies of size $N$:
\begin{equation*}
R_C f = \sum_{n \in \Z^{\nu}} \chi_C(n) e^{i \langle n \rangle_\omega \cdot x} \hat{f}(\langle n \rangle_\omega), \quad P_N f = \sum_{n \in \Z^{\nu}} \chi_N(\langle n \rangle_\omega) e^{i \langle n \rangle_\omega \cdot x} \hat{f}(\langle n \rangle_\omega).
\end{equation*}
Above $\chi_A$ denotes a function, which smoothly localizes to size $A$.

\smallskip

To obtain an $L^4$-Strichartz estimate for quasi-periodic functions, we need to count the number of frequencies $\langle n \rangle_\omega$ contained in an interval $I$ of length $\sim 1$: A counting argument based on the rank of the lattice gives
\begin{equation*}
\# \{ n \in \Z^{\nu}: |n| \sim C, \; \langle n \rangle_\omega \in I \} \lesssim C^{\nu-1}.
\end{equation*}
%An alternative estimate follows from a diophantine condition on $\omega \in \R^{\nu}_{>0}$, which reads as follows:
%
%\medskip
%
%\emph{There are $\alpha>0$, $0 \leq \beta < \nu -1$ such that for $n \in \Z^{\nu}$, $n \neq 0$ it holds}
%\begin{equation}
%\label{eq:DiophantineConditionIntro}
%|\langle n \rangle_\omega| \geq \alpha |n|^{-\beta}.
%\end{equation}
%
%\smallskip

%Note that the counting argument makes the diophantine condition sensible only for $0 \leq \beta < \nu-1$, whereas, in the non-periodic case $\nu \geq 2$, diophantine approximation gives the lower bound $\beta \geq 1$. 
We define the density parameter as $b=\nu-1$.
% as follows:
%\begin{equation}
%\label{eq:DensityParameter}
%b = 
%\begin{cases}
%\beta, \quad \quad \eqref{eq:DiophantineConditionIntro} \text{ holds}, \\
%\nu-1, \quad \text{ else}.
%\end{cases}
%\end{equation}

\medskip

The generalization of the estimate \eqref{eq:PeriodicL4StrichartzIntroduction} reads
\begin{equation}
\label{eq:QPStrichartzEstimateIntroduction}
\| e^{it \partial_x^2} R_C f \|_{L_t^4([0,1],\mathcal{L}^4_x)} \lesssim_\varepsilon C^{\frac{b}{4}+\varepsilon}  \| f \|_{\mathcal{L}^2_x}.
\end{equation}
The above estimate is obtained in two steps:
\begin{itemize}
\item[(i)] The decoupling or square function estimate provides us with an almost orthogonal frequency decomposition of $\langle n \rangle_\omega$ into unit intervals. 
\item[(ii)] In the second step we use a bound for the height $|n| \leq C$ and a counting argument to estimate the number of frequencies $\langle n \rangle_\omega$ contained in a unit interval.
\end{itemize}
The first step trivializes the time-evolution on a scale determined by the frequency $\langle n \rangle_\omega$, but we can have many frequencies with a height $| n  | \gg |\langle n \rangle_{\omega}|$. This creates a mismatch between the almost orthogonal decomposition provided by the square function estimate and the number of frequencies within the intervals, which is divergent from the periodic case. Still it appears that the arguments optimally take advantage of the time-oscillation and the density of frequencies. We give a simple example, which shows sharpness of \eqref{eq:QPStrichartzEstimateIntroduction}. 

\medskip

 We show the following:
\begin{theorem}
\label{thm:L4StrichartzFixedTime}
Let $T \in (0,1]$, and $\sigma(u_0) \subseteq \Lambda = \omega \cdot \Z^{\nu}$ with $\omega \in \R^{\nu}_{>0}$ non-resonant, and $b=\nu-1$. Then the following estimates hold:
\begin{align}
\label{eq:FixedTimeL4StrichartzSEQIntro}
\| e^{it \partial_x^2} u_0 \|_{L_t^4([0,T],\mathcal{L}^4_x(\R))} &\lesssim T^{\frac{1}{8}} \| u_0 \|_{\mathcal{H}^s_{\Lambda}} \text{ for } s > \frac{b}{4}, \\
\label{eq:FixedTimeL4StrichartzAiryIntro}
\| e^{t \partial_x^3} P_N R_C u_0 \|_{L_t^4([0,T],\mathcal{L}^4_x(\R))} &\lesssim T^{\frac{1}{4}} \langle T^{-\frac{1}{2}} N^{-\frac{1}{2}} \rangle^{\frac{1}{4}} C^{\frac{b}{4}} \| u_0 \|_{\mathcal{L}^2_x} \text{ for } 1 \leq N \lesssim C.
\end{align}
\end{theorem}

\medskip
 
The arguments can clearly cover more general dispersion relations. The above estimates extend the Strichartz estimates for periodic solutions to one-dimensional Schr\"odinger and Airy equations due to Bourgain \cite{Bourgain1993A,Bourgain1993B}. The higher-dimensional case, also on irrational tori was settled as a consequence of sharp $\ell^2$-decoupling due to Bourgain--Demeter \cite{BourgainDemeter2015}. We rely on decoupling to show Strichartz estimates for quasi-periodic functions in higher dimensions; see Theorem \ref{thm:FixedTimeStrichartzGeneral}. 
%Here we can obtain the density of frequencies by imposing diophantine conditions in every direction.
 
\medskip

A different type of Strichartz estimate for quasi-periodic functions with additional time averaging was proved last year by Klaus \cite{Klaus2023}:
\begin{equation*}
\| e^{it \partial_x^2} f \|_{\mathcal{L}_{t,x}^4(\R \times \R)} \lesssim \| f \|_{\mathcal{L}^2_x(\R)}.
\end{equation*} 
He considered the special case of quasi-periodic functions $f \in \mathcal{B}_2(\R)$ with $\sigma(f) \subseteq \Z + \sqrt{2} \Z$. The proof relied on a counting argument similar to the ones from Bourgain \cite{Bourgain1993A}.
%The first Strichartz estimate for quasiperiodic functions was proved by Klaus \cite{Klaus2023}:
%\begin{equation*}
%\| e^{it \partial_x^2} u_0 \|_{\mathcal{L}^4_{t,x}(\R \times \R)} \lesssim \| u_0 \|_{\mathcal{L}_x^2(\R)}.
%\end{equation*}
%He considered $\sigma(u_0) \subseteq \Z + \sqrt{2} \Z$ with Fourier coefficients absolutely summable.
% The absolute summability can be viewed as a technical condition to state the results: absolute summability keeps the $\mathcal{L}_{t,x}^4$-norm finite. 
With additional time averaging, the almost orthogonal decomposition in frequencies $\langle n \rangle_\omega$ can be taken arbitrarily small. We show the following Strichartz estimates for Besicovitch-almost periodic functions:
\begin{theorem}
\label{thm:AveragedL4Strichartz}
Let $u_0 \in \mathcal{B}_2(\R)$. The following Strichartz estimates hold:
\begin{align}
\label{eq:L4StrichartzSEQ}
\| e^{it \partial_x^2} u_0 \|_{\mathcal{L}^4_{t,x}(\R \times \R)} &\lesssim \| u_0 \|_{\mathcal{L}_x^2(\R)}, \\
\label{eq:L4StrichartzAiry}
\| e^{t \partial_x^3} u_0 \|_{\mathcal{L}^4_{t,x}(\R \times \R)} &\lesssim \| u_0 \|_{\mathcal{L}_x^2(\R)}.
\end{align}
\end{theorem}

Klaus explained why \eqref{eq:L4StrichartzSEQ} does not imply local well-posedness of \eqref{eq:CubicNLSIntro} in $\mathcal{L}_x^2$. The problem is that applying the $TT^*$-argument we find with
\begin{equation*}
T: \mathcal{L}_x^2(\R) \to \mathcal{L}^4_{t,x}(\R \times \R), \quad u_0 \mapsto e^{it \partial_x^2} u_0
\end{equation*}
that
\begin{equation*}
T^*: \mathcal{L}^{\frac{4}{3}}_{t,x}(\R \times \R) \to \mathcal{L}^2_x(\R), \quad F \mapsto \lim_{L \to \infty} \frac{1}{2L} \int_{[-L,L]} e^{-is \partial_x^2} F(s) ds.
\end{equation*}
Consequently,
\begin{equation*}
TT^*: \mathcal{L}^{\frac{4}{3}}_{t,x}(\R \times \R) \to \mathcal{L}^4_{t,x}(\R \times \R), \quad F \mapsto \lim_{L \to \infty} \frac{1}{2L} \int_{[-L,L]} e^{i(t-s) \partial_x^2} F(s) ds.
\end{equation*}
This $TT^*$-operator is not suitable for estimating the Duhamel formula:
\begin{equation*}
u(t) = e^{it \partial_x^2} u_0 - i \int_0^t e^{i(t-s) \partial_x^2} (|u|^2 u)(s) ds.
\end{equation*}

\medskip

It is the preceding observation, which suggests to study the Strichartz estimates from Theorem \ref{thm:L4StrichartzFixedTime}. We remark that the Strichartz estimates from Theorem \ref{thm:AveragedL4Strichartz} extend Strichartz estimates for periodic functions, too: Let $f \in L^2(\T)$. Then $e^{it \partial_x^2} f$ is time-periodic, and we have
\begin{equation*}
\| e^{it \partial_x^2} f \|_{\mathcal{L}^4_{t,x}(\R^2)} = \| e^{it \partial_x^2} f \|_{L^4_{t,x}(\T^2)} \lesssim \| f \|_{\mathcal{L}^2_x(\R)} = \| f \|_{L^2(\T)},
\end{equation*}
which recovers \eqref{eq:PeriodicL4StrichartzIntroduction}.

 Applying $\ell^2$-decoupling as a substitute for the square function estimate, we show Strichartz estimates in higher dimensions with additional time averaging for quasi-periodic functions in $\mathcal{H}^\varepsilon_\Lambda$ in Theorem \ref{thm:AveragedStrichartzHigh} with frequency-vector satisfying a diophantine condition.

\medskip

Finally, we apply the Strichartz estimates to show new low regularity well-posedness of nonlinear dispersive equations with quasi-periodic initial data. 
%As recalled above, the Strichartz estimates with time averaging are not readily applicable. Instead, we use $L^p$-estimates on fixed time intervals, but still averaging in space. Here we need bilinear refinements.
%We shall see how the almost orthogonality can be used to obtain estimates on fixed-time intervals, which improve on Bernstein's inequality. We focus for the sake of simplicity on the one-dimensional case with frequency set of the initial data contained in $\Z + \sqrt{2} \Z$: In the following for $n \in \Z^2$ we write $\langle n \rangle_\omega = n_1 + \sqrt{2} n_2$. We define
%\begin{equation*}
%QP(1,\sqrt{2}) = \overline{\{ f = \sum_{n \in \Z^2} e^{i \langle n \rangle_\omega x} \hat{f}(n) \, : \, \hat{f}(n) \neq 0 \text{ for finitely many } n \}}^{\| \cdot \|_{\mathcal{L}^2_x}}.
%\end{equation*}
%Let $N \in 2^{\N_0}$. We write $P_N$ for the frequency projection to $\{ n \in \Z^2 : |\langle n \rangle_\omega| \sim N \} = A_N$:
%\begin{equation*}
%P_N f= \sum_{n \in A_N} e^{i \langle n \rangle_\omega x} \hat{f}(n).
%\end{equation*}
%and $R_N$ for the frequency projection to the set $\{ n \in \Z^2 : \langle n \rangle \sim N \} = B_N$:
%\begin{equation*}
%R_N f = \sum_{n \in B_N} e^{i \langle n \rangle_\omega x} \hat{f}(n).
%\end{equation*}
%\smallskip
%
%This leads us to the definition of the following Sobolev-type spaces:
%\begin{equation*}
%\begin{split}
%\overline{\mathcal{H}}^s_x &= \{ f \in QP(1,\sqrt{2}) : \| f \|_{\mathcal{H}^s_x} < \infty \}, \\
%\| f \|^2_{\mathcal{H}^s_x} &= \sum_{N \in 2^{\N_0}} N^{2s} \| R_N f \|_{\mathcal{L}^2_x}^2.
%\end{split}
%\end{equation*}
We show the following low regularity local well-posedness result for the nonlinear Schr\"o\-dinger equation:
\begin{equation}
\label{eq:NLSIntroII}
\left\{ \begin{array}{cl}
i \partial_t u + \partial_{x}^2 u &= \pm |u|^2 u , \quad (t,x) \in \R \times \R, \\
u(0) &= f \in \mathcal{H}^s_{\Lambda}(\R).
\end{array} \right.
\end{equation}
By local well-posedness we refer to existence, uniqueness, and continuous dependence of the solution on the initial data.

\begin{theorem}[Local well-posedness of NLS with quasi-periodic initial data]
\label{thm:LWPNLS}
Let $\nu \in \N$, and suppose that $\omega \in \R^{\nu}_{>0}$ is non-resonant. Let $\Lambda = \omega \cdot \Z^{\nu}$ with $b=\nu-1$. Then \eqref{eq:NLSIntroII} is locally well-posed for $s > \frac{b}{2}$.
\end{theorem}

Previously, Oh \cite{Oh2015} considered \eqref{eq:NLSIntroII} with $\Lambda = \Z + \sqrt{2} \Z$ and showed local well-posedness for $s > 1$. More precisely, he showed local well-posedness in $\mathcal{A}_\omega$. Applying the Cauchy-Schwarz inequality yields the following embedding for $s>1$:
\begin{equation*}
\mathcal{H}^s_{\Lambda} \subseteq \mathcal{A}_\omega.
\end{equation*}

%We have the following:
%\begin{corollary}
%Let $\Lambda = \Z + \sqrt{2} \Z$. Then \eqref{eq:NLSIntroII} is locally well-posed for $s>\frac{1}{2}$.
%\end{corollary}
For $\Lambda = \Z + \sqrt{2} \Z$ Theorem \ref{thm:LWPNLS} yields local well-posedness for $s>\frac{1}{2}$. We regard this as low regularity local well-posedness since the space $\mathcal{H}^s_\Lambda$ is not a Banach algebra for $s < 1$. 
The improvement of the local well-posedness rests on the fixed-time estimates \eqref{eq:FixedTimeL4StrichartzSEQIntro}.
The key bilinear estimates are a consequence of almost orthogonality and Galilean invariance, which read for $\Lambda = \Z + \sqrt{2} \Z$:
\begin{equation*}
\| R_{C_1} e^{it \partial_x^2} f_1 R_{C_2} e^{it \partial_x^2} f_2 \|_{L_t^2([0,T],\mathcal{L}^2_x(\R))} \lesssim_\varepsilon T^{\frac{1}{4}} C_{\min}^{\frac{1}{2}+\varepsilon} \| R_{C_1} u \|_{\mathcal{L}^2_x} \| R_{C_2} v \|_{\mathcal{L}^2_x}.
\end{equation*}
With the above at hand, it is standard to apply the contraction mapping principle in adapted function spaces \cite{KochTataru2005,HadacHerrKoch2009,HadacHerrKoch2009Erratum}. For this reason we shall be brief. We remark that, since the solutions are constructed via the contraction mapping principle, the dependence of solutions on initial data is real analytic like the nonlinearity. This result is sharp up to endpoints. We show $C^3$-ill-posedness of \eqref{eq:CubicNLSIntro} on $\Lambda = \omega \cdot \Z^{\nu}$ in $\mathcal{H}^s_\Lambda$ for $s<\frac{\nu-1}{2}$.
Moreover, the arguments are flexible again and apply to Schr\"odinger equations with algebraic power nonlinearity. We record a result for the NLS with algebraic power- nonlinearity in Section \ref{section:LWP}: Let $\Lambda = \omega \cdot \Z^{\nu} \subseteq \R$ be like above and $m \geq 3$. Then
\begin{equation*}
\left\{ \begin{array}{cl}
i \partial_t u + \partial_x^2 u &= \pm |u|^{2(m-1)} u, \quad (t,x) \in \R \times \R, \\
u(0) &= u_0 \in \mathcal{H}^s_\Lambda
\end{array} \right.
\end{equation*}
is locally well-posed for $s>s^*=2(\nu-1)\big( \frac{1}{2} - \frac{1}{2m} \big) + (1-\frac{3}{m})$. Damanik \emph{et al.} \cite{DamanikLiXu2024A} show local well-posedness for coefficients decaying like $|c(n)| \lesssim (1+|n|)^{-r}$ for $r > 4 \nu$. Theorem \ref{thm:LWPGenearlNLSDetailed} implies an admissible Fourier decay of $r>s^* + \frac{\nu}{2}$, which constitutes an improvement.

\medskip

Finally, we remark that one can understand the ill-posedness result as limitation of the Harmonic Analysis approach presently taken. Proving global well-posedness of \eqref{eq:CubicNLSIntro}, which would be very natural in $\mathcal{L}^2_x(\R)$, is beyond the methods of the paper. Global well-posedness might also depend on the lattice $\Lambda$ and whether the evolution is focusing or defocusing. The present analysis suggests that a proof of global well-posedness can only be accomplished employing the complete integrability. This would parallel the result for decaying data, in the sense that local well-posedness below $L^2(\R)$ could only be proved using the complete integrability \cite{HarropGriffithsKillipVisan2024}.

\medskip

\emph{Outline of the paper.} In Section \ref{section:Preliminaries} we elaborate on the properties of almost-periodic functions and the diophantine conditions. Moreover, we recall the C\'ordoba--Fefferman square function estimate and $\ell^2$-decoupling. In Section \ref{section:FixedTimeEstimates} Strichartz estimates on fixed time scales are proved, in Section \ref{section:AveragedStrichartzEstimates} we show Stri\-chartz estimates with additional averaging in time. Finally, in Section \ref{section:LWP} we show local well-posedness results for Schr\"odinger equations with quasi-periodic initial data. Examples pointing out sharpness of the Strichartz estimates and of the local well-posedness of the cubic NLS are given in Section \ref{section:Examples}. In the Appendix we indicate how the analysis recovers Tsugawa's result for the KdV equation.

\medskip

\textbf{Basic notations:}

\begin{itemize}
\item For $A \in \C^d$ we denote the Euclidean norm with $|A| = \sqrt{|A_1|^2 + \ldots + |A_d|^2}$. We let $\langle A \rangle = 1 + |A|$. For $A,B \in \R$ let $A \wedge B = \min(A,B)$, and $A \vee B = \max(A,B)$.
\item $A \lesssim B$ denotes the inequality $A \leq C B$ with a harmless constant $C$. Dependence of $C$ on parameters is indicated with subscripts, e.g., $A \lesssim_\varepsilon B$ denotes $A \leq C(\varepsilon) B$ with $C$ depending on $\varepsilon$.
\item Dyadic numbers are denoted with capital letters $M,N, \ldots \in 2^{\N_0}$.
\item For $A \subseteq \R^d$ we denote with $\mathcal{N}_\delta(A)$ the $\delta$-neighborhood of $A$.
\item For $1 \leq p < \infty$, $A \subseteq \R^d$, the usual $L^p$-norm for Lebesgue-measurable functions $f:A \to \C$  is given by
\begin{equation*}
\| f \|_{L^p(A)}^p = \int_A |f(x)|^p dx.
\end{equation*}
For $p = \infty$ the $L^\infty$-norm describes the essential supremum.
\item With $\mathcal{L}^p_x$ we denote the norms on $\R^d$ after averaging, see Section \ref{section:Preliminaries}.

\end{itemize}

\section{Preliminaries}
\label{section:Preliminaries}
\subsection{Almost-periodic functions}

In this section we recall key properties of Besico\-vitch-almost periodic functions. We refer to \cite{Corduneanu1989,BesicovitchBohr1931} for further reading.
For $1 \leq p < \infty$, and a measurable function $f: \R^d \to \C$ define
\begin{equation*}
\| f \|^p_{\mathcal{L}^p_x(\R^d)} = \lim_{L \to \infty} \frac{1}{(2L)^d} \int_{[-L,L]^d} |f(x)|^p dx.
\end{equation*}
Moreover, let
\begin{equation*}
\| f \|_{\mathcal{L}^\infty(\R^d)} := \| f \|_{L^\infty(\R^d)}.
\end{equation*}

We define Besicovitch-almost-periodic functions in higher dimensions as
\begin{equation*}
\mathcal{B}_2(\R^d) = \overline{\{ f(x) = \sum_{\lambda \in \Lambda_f} a_\lambda e^{i \lambda x}: \Lambda_f \subseteq \R^d \text{ finite }  \}}^{\| \cdot \|_{\mathcal{L}_x^2}}.
\end{equation*}

The Fourier expansion is valid in $\mathcal{L}^2_x(\R^d)$:
\begin{equation*}
f(x) = \sum_{\lambda \in \Lambda} \hat{f}(\lambda) e^{i \lambda x}, \quad x \in \R^d
\end{equation*}
with $\Lambda \subseteq \R^d$ at most countable infinite.\footnote{Precisely, $\mathcal{B}_2$ is comprised of equivalence classes of functions $f \sim g$ with $\mathcal{M}(|f-g|^2) = 0$.} By H\"older's inequality, we obtain the Fourier coefficients by
\begin{equation*}
\hat{f}(\lambda) = \lim_{L \to \infty} \frac{1}{2L} \int_{[-L,L]} f(x) e^{-i \lambda x} dx,
\end{equation*}
which satisfies $|\hat{f}(\lambda)| \leq \| f \|_{\mathcal{L}^2}$. Define the mean-value by
\begin{equation*}
\mathcal{M}(f) = \lim_{L \to \infty} \frac{1}{2L} \int_{[-L,L]} f(x) dx.
\end{equation*}

In the following let $\sigma(f) \subseteq \Lambda \subseteq \R^d$. We have the following classification of $f \in \mathcal{B}_2(\R^d)$:
\begin{itemize}
\item If $\Lambda$ can be chosen such that $\Lambda \subseteq \omega^{(1)} \Z \times \ldots \times \omega^{(d)} \Z$ for some $\omega^{(i)} > 0$, $i=1,\ldots,d$, then $f$ is a periodic function.
\item If $\Lambda$ can be chosen such that $\Lambda \subseteq \omega^{(1)} \cdot \Z^{\nu_1} \times \ldots \times \omega^{(d)} \cdot \Z^{\nu_d}$ with $\omega^{(i)} \in \R^{\nu_i}_{>0}$ non-resonant, then $f$ is a quasi-periodic function.
\item Else, $\Lambda \subseteq \sum_{k_1=1}^\infty \omega^{(1)}_k \Z \times \sum_{k_2=1}^\infty \omega^{(2)}_k \Z \ldots \times \sum_{k_d=1}^\infty \omega^{(d)}_k \Z$ for $(\omega^{(i)}_k)_{k \in \N} \subseteq \R_{>0}$ linearly independent over $\Q$ for any $i=1,\ldots,d$, then $f$ is an almost periodic function.
\end{itemize}

Let $f \in \mathcal{B}_2(\R)$ be a quasi-periodic function. Then, there are rationally independent $\{ \omega_1,\ldots,\omega_{\nu} \} \subseteq \R_{>0}$ such that $\sigma(f) \subseteq \Lambda = \omega \cdot \Z^{\nu}$. For $n \in \Z^{\nu}$ we let $\langle n \rangle_\omega = n \cdot \omega$. For any frequency $\lambda \in \Lambda$ there is a unique $n \in \Z^{\nu}$ such that $\lambda = \langle n \rangle_\omega$. We refer to $|n|$ as the \emph{height} of the frequency $\langle n \rangle_\omega$.  We have the Fourier series representation with convergence in $\mathcal{L}^2_x(\R)$:
\begin{equation*}
f(x) = \sum_{n \in \Z^{\nu}} e^{i \langle n \rangle_\omega x} \hat{f}(\langle n \rangle_\omega).
\end{equation*}
We quantify regularity through the Sobolev-type norms: For $s \geq 0$ define
\begin{equation*}
\mathcal{H}^s_{\Lambda}(\R) = \{ f \in \mathcal{B}_2(\R) : \| f \|_{\mathcal{H}^s_\Lambda} < \infty \}, \quad \| f \|^2_{\mathcal{H}^s_\Lambda} = \sum_{n \in \Z^{\nu}} \langle n \rangle^{2s} | \hat{f}(\langle n \rangle_\omega)|^2.
\end{equation*}

%The Strichartz estimates for quasiperiodic functions without averaging in time are formulated with the aid of a diophantine condition on $\omega \in \R^{\nu}_{>0}$, which reads as follows:
%
%\smallskip
%
%\emph{There are $\alpha>0$, $0<\beta<\nu-1$ such that for $n \in \Z^{\nu}$, $n \neq 0$ it holds}
%\begin{equation}
%\label{eq:DiophantineConditionPrelim}
%|\langle n \rangle_\omega| \geq \alpha |n|^{-\beta}.
%\end{equation}
%
%\smallskip

The density parameter is defined as $b=\nu-1$.
%\begin{equation}
%\label{eq:DensityParameterSection}
%b = 
%\begin{cases}
%\beta, \quad \quad \eqref{eq:DiophantineConditionPrelim} \text{ holds}, \\
%\nu-1, \quad \text{ else}.
%\end{cases}
%\end{equation}

\medskip

The generalization of the above considerations to higher-dimensional quasi-perio\-dic functions is straight-forward. Let $d \in \N$, $(\nu_1,\ldots,\nu_d) \in \N^d$, $\omega^{(i)} \in \R_{>0}^{\nu_i}$ for $i=1,\ldots,d$ with $\omega^{(i)}$ non-resonant, and let
\begin{equation}
\label{eq:HigherDimensionalLattice} 
\Lambda = \omega^{(1)} \cdot \Z^{\nu_1} \times \ldots \times \omega^{(d)} \cdot \Z^{\nu_d}.
\end{equation}
Define for $n_1 \in \Z^{\nu_1},\ldots,n_d \in \Z^{\nu_d}$: 
\begin{equation*}
\langle n \rangle_\omega = (\langle n_1 \rangle_{\omega^{(1)}}, \ldots, \langle n_d \rangle_{\omega^{(d)}}) \in \R^d.
\end{equation*}
 We have the Fourier series expansion in $\mathcal{L}^2(\R^d)$:
\begin{equation*} 
f = \sum_{n_1 \in \Z^{\nu_1},\ldots,n_d \in \Z^{\nu_d} }  e^{i \langle n \rangle_\omega \cdot x} \hat{f}(\langle n \rangle_\omega)
\end{equation*}
and the Sobolev norm:
\begin{equation*}
\| f \|^2_{\mathcal{H}^s_{\Lambda}(\R^d)} = \sum_{n_1 \in \Z^{\nu_1}, \ldots, n_d \in \Z^{\nu_d}} \langle n \rangle^{2s} |\hat{f}(\langle n \rangle_\omega)|^2.
\end{equation*}
We let $b_i = \nu_i - 1$ and $b=b_1+\ldots+b_d$.

%The higher-dimensional diophantine condition on $\omega^{(1)}$, $\ldots$, $\omega^{(d)}$ reads: 
%
%\medskip
%
%\emph{There are $\alpha_i>0$, $\beta_i \geq 0$, $i=1,\ldots,d$ such that}
%\begin{equation}
%\label{eq:DiophantineHigherDimensions}
%|n_1 \omega^{(1)}| \geq \alpha_1 |n_1|^{-\beta_1}, \ldots, \; |n_d \omega^{(d)}| \geq \alpha_d |n_d|^{-\beta_d}.
%\end{equation}
%
%The density parameters are defined by
%\begin{equation}
%\label{eq:DensityParameterHigh}
%b_i = 
%\begin{cases}
%\beta_i, \quad \quad \eqref{eq:DiophantineHigherDimensions} \text{ holds}, \\
%\nu_i-1, \quad \text{ else}.
%\end{cases}
%\end{equation}

\subsection{C\'ordoba--Fefferman square function estimate}

We recall the following square function estimate \cite{Fefferman1973,Cordoba1979,Cordoba1982} as one of the key ingredients for the $L^4$-Strichartz estimates in one dimension:
\begin{theorem}[C\'ordoba--Fefferman~square~function~estimate]
\label{thm:CFSquareFunctionEstimateCurvature}
Let $\Gamma = \{(\xi,h(\xi)): \xi \in [-1,1] \}$ with 
\begin{equation*}
\frac{1}{2} \leq h''(\xi) \leq 2
\end{equation*}
and $\text{supp}(\hat{F}) \subseteq \mathcal{N}_\delta(\Gamma)$ for some $0<\delta \ll 1$. Let  $\Theta_\delta$ denote a cover of $\mathcal{N}_\delta(\Gamma)$ with finitely overlapping $10 \delta^{\frac{1}{2}} \times 10 \delta$-boxes with long-side pointing  into tangential and short side pointing into normal direction. Then the following estimate holds:
\begin{equation*}
\| F \|_{L^4(\R^2)} \lesssim \big\| \big( \sum_{\theta \in \Theta_\delta} |F_\theta|^2 \big)^{\frac{1}{2}} \big\|_{L^4(\R^2)}.
\end{equation*}
with $F_\theta$ denoting the Fourier projection of $F$ to $\theta$.
\end{theorem}

This will separate the frequencies for the Schr\"odinger evolution and trivialize the time-evolution. Note that for $h(\xi) = \xi^3$ the curvature degenerates at the origin. We show the following variant directly:
\begin{theorem}[C\'ordoba--Fefferman square function estimate for the cubic]
\label{thm:CFSquareFunctionCubic}
Let $\Gamma = \{(\xi,\xi^3) : \xi \in [-1,1] \}$ and $\text{supp}(\hat{F}) \subseteq \mathcal{N}_\delta(\Gamma)$ for some $0<\delta\ll 1$. Let $\Theta_\delta$ denote a cover of $\mathcal{N}_\delta(\Gamma)$ with finitely overlapping $10 \delta^{\frac{1}{3}} \times 10 \delta$-boxes with long side pointing into tangential and short side into normal direction. Then the following estimate holds:
\begin{equation*}
\| F \|_{L^4(\R^2)} \lesssim \big\| \big( \sum_{\theta \in \Theta_\delta} |F_\theta|^2 \big)^{\frac{1}{2}} \big\|_{L^4(\R^2)}.
\end{equation*}
\end{theorem}
\begin{proof}
We split the Fourier support into positive and negative frequencies and suppose in the following by symmetry that
\begin{equation*}
\text{supp}(\hat{F}) \subseteq \mathcal{N}_\delta (\{ (\xi,\xi^3) : \, \xi \in [0,1] \}).
\end{equation*}
We decompose $F = \sum_{\theta \in \Theta_\delta} F_\theta$ and obtain by Plancherel's theorem
\begin{equation*}
\begin{split}
\int_{\R^2} \big| \sum_{\theta \in \Theta_\delta} F_\theta \big|^4 dx &= \int_{\R^2} \big| \sum_{\theta_1 \in \Theta_\delta} F_{\theta_1} \sum_{\theta_2 \in \Theta_\delta} F_{\theta_2} \big|^2 \\
&= \int_{\R^2} \big| \sum_{\theta_1} \widehat{F_{\theta_1}} * \sum_{\theta_2} \widehat{F_{\theta_2}} \big|^2 \\
&= \int_{\R^2} \big( \sum_{\theta_1} \widehat{F_{\theta_1}} * \sum_{\theta_2} \widehat{F_{\theta_2}} \big) \overline{\big( \sum_{\theta_3} \widehat{F_{\theta_3}} * \sum_{\theta_4} \widehat{F_{\theta_4}} \big)}.
\end{split} 
\end{equation*}
Non-trivial contribution to the integral stems from solutions to the system
\begin{equation*}
\left\{ \begin{array}{cl}
\xi_1 + \xi_2 &= \xi_3 + \xi_4, \\
\xi_1^3 + \xi_2^3 &= \xi_3^3 + \xi_4^3 + \mathcal{O}(\delta)
\end{array} \right.
\end{equation*}
for $(\xi_i,\xi_i^3) \in \theta_i$. We shall establish a biorthogonality: For $\xi_1 + \xi_2 \ll \delta^{\frac{1}{3}}$ it is clear that $\xi_i$ are coming from the first $\mathcal{O}(1)$-blocks $\delta^{\frac{1}{3}} \times \delta$-close to the origin. So we suppose that $\xi_1 + \xi_2 \gtrsim \delta^{\frac{1}{3}}$. We take the cubic power of the first equation:
\begin{equation*}
\xi_1^3 + 3 \xi_1^2 \xi_2 + 3 \xi_1 \xi_2^2 + \xi_2^3 = \xi_3^3 + 3 \xi_3^2 \xi_4 + 3 \xi_3 \xi_4^2 + \xi_4^3.
\end{equation*}
Subtracting this from the second equation gives
\begin{equation*}
3\xi_1 \xi_2 (\xi_1+\xi_2) = 3 \xi_3 \xi_4 (\xi_3 + \xi_4) + \mathcal{O}(\delta).
\end{equation*}
This implies
\begin{equation*}
\xi_1 \xi_2 = \xi_3 \xi_4 + \mathcal{O}(\delta^{\frac{2}{3}}).
\end{equation*}
We square the first equation and subtract the above display to find
\begin{equation*}
\left\{ \begin{array}{cl}
\xi_1 + \xi_2 &= \xi_3 + \xi_4, \\
\xi_1^2 + \xi_2^2 &= \xi_3^2 + \xi_4^2 + \mathcal{O}(\delta^{\frac{2}{3}}).
\end{array} \right.
\end{equation*}
Now the claim follows from the biorthogonality underpinning Theorem \ref{thm:CFSquareFunctionEstimateCurvature}. For the sake of self-containedness we give the details. Write
\begin{equation*}
(\xi_1 - \xi_3)(\xi_1 + \xi_3) = (\xi_4-\xi_2)(\xi_4+\xi_2) + \mathcal{O}(\delta^{\frac{2}{3}}).
\end{equation*}
If $|\xi_1-\xi_3| \lesssim \delta^{\frac{1}{3}}$ biorthogonality follows. So, we suppose that $|\xi_1-\xi_3| \gg \delta^{\frac{1}{3}}$. In this case we find
\begin{equation*}
\left\{ \begin{array}{cl}
\xi_1 + \xi_3 &= \xi_2 + \xi_4 + \mathcal{O}(\delta^{\frac{1}{3}}), \\
\xi_1 + \xi_2 &= \xi_3 + \xi_4,
\end{array} \right.
\end{equation*}
which implies $|\xi_1-\xi_4| \lesssim \delta^{\frac{1}{3}}$, hence biorthogonality.
\end{proof}

\subsection{$\ell^2$-decoupling}

Here we recall the $\ell^2$-decoupling result due to Bour\-gain-\-De\-meter \cite[Theorem~1]{BourgainDemeter2015} for future use. Let $h \in C^2(B_d(0,1),\R)$, $\Gamma = \{(\xi,h(\xi)) : \xi \in B_d(0,1) \} \subseteq \R^{d+1}$ be a compact elliptic hypersurface with principal curvature $K_i \in [C^{-1},C]$ for some $C>0$. Let $p_d = \frac{2(d+2)}{d}$, and define
\begin{equation*}
\alpha(p) = \begin{cases}
0, &\quad 2 < p < p_d, \\
\frac{d}{2} - \frac{d+2}{p}, &\quad p_d \leq p < \infty.
\end{cases}
\end{equation*}

\begin{theorem}[$\ell^2$-decoupling]
\label{thm:Decoupling}
Let $0 < \delta \ll 1$, $F \in \mathcal{S}(\R^{d+1})$ with $\text{supp}(\hat{F}) \subseteq \mathcal{N}_\delta(\Gamma)$ and $2 < p < \infty$. Then the following estimate holds:
\begin{equation*}
\| F \|_{L^{p}(B_{\delta^{-1}})} \lesssim_\varepsilon \delta^{-(\alpha(p)+\varepsilon)} \big( \sum_{\theta \in \Theta_\delta} \| F_\theta \|^2_{L^{p}(w_{B_{\delta^{-1}}})} \big)^{\frac{1}{2}}.
\end{equation*}
In the above display $\Theta_\delta$ contains a finitely overlapping collection of caps of size comparable to $\delta^{\frac{1}{2}} \times \ldots \times \delta^{\frac{1}{2}} \times \delta$ with small side pointing into the normal direction, which covers $\mathcal{N}_\delta(\Gamma)$.
\end{theorem}

Recall that the above decoupling also implies the global estimate with $L^p$-norm taken over the whole space. We refer to \cite{Demeter2020} for a textbook treatment.

\section{Strichartz estimates on fixed time intervals}
\label{section:FixedTimeEstimates}
This section is devoted to the proof of Strichartz estimates on fixed time intervals. We handle the $L^4$-estimates in one dimension  from Theorem \ref{thm:CFSquareFunctionEstimateCurvature} first. In the $L^4$-estimate we obtain a gain for estimates on small time scales.

\subsection{The $L^4$-estimates}

Recall that the frequencies are contained in the lattice $\Lambda = \omega_1 \Z + \ldots + \omega_{\nu} \Z$
with $\nu \in \N$, $\omega \in \R^{\nu}_{>0}$ non-resonant, and let $b=\nu-1$.

In the following we prove the estimates
\begin{align}
\label{eq:FixedTimeL4StrichartzSEQ}
\| e^{it \partial_x^2} u_0 \|_{L_t^4([0,T],\mathcal{L}^4_x(\R))} &\lesssim T^{\frac{1}{8}} \| u_0 \|_{\mathcal{H}^s_x} \text{ for } s > \frac{b}{4}, \\
\label{eq:FixedTimeL4StrichartzAiry}
\| e^{t \partial_x^3} P_N R_C u_0 \|_{L_t^4([0,T],\mathcal{L}^4_x(\R))} &\lesssim T^{\frac{1}{4}} C^{\frac{b}{4}} \langle T^{-\frac{1}{2}} N^{-\frac{1}{2}} \rangle^{\frac{1}{4}} \| u_0 \|_{\mathcal{L}^2_x} \text{ for } 1 \leq N \lesssim C.
\end{align}
\begin{proof}[Proof~of~Theorem~\ref{thm:L4StrichartzFixedTime}]
We begin with the proof of \eqref{eq:FixedTimeL4StrichartzSEQ}. Note that the height bounds the modulus of frequencies $\langle n \rangle_\omega$, which can occur. By the triangle inequality we find that it suffices to prove for $1 \leq N \lesssim C$:
\begin{equation*}
\| e^{it \partial_x^2} P_N R_C f \|_{L_t^4([0,T],\mathcal{L}_x^4)} \lesssim T^{\frac{1}{8}} C^{\frac{b}{4}} \| f \|_{\mathcal{L}^2_x}.
\end{equation*}
For $T \geq N^{-2}$ we shall show that
\begin{equation}
\label{eq:AlmostOrthogonalityL4Strichartz}
\| e^{it \partial_x^2} P_N R_C f \|_{L_t^4([0,T],\mathcal{L}^4_x)} \lesssim \big( \sum_{I_T: T^{-\frac{1}{2}}-\text{interval}} \| e^{it \partial_x^2} P_N R_C P_{I_T} f \|^2_{L_t^4(w_T, \mathcal{L}^4_x)} \big)^{\frac{1}{2}}.
\end{equation}
Above $L_t^4(w_T)$ denotes the $L^4$-norm with a polynomial weight $w_T(t) \lesssim (1+ |t|/T)^{-100}$ decaying off $[-T,T]$:
\begin{equation*}
\| f \|_{L_t^4(w_T)}^4 = \int_{\R} |f(t)|^4 w_T(t) dt.
\end{equation*}

Taken the above estimate for granted, it remains to prove an estimate
\begin{equation}
\label{eq:AuxL4SEQI}
\| e^{it \partial_x^2 } R_C P_{I_T} f \|_{L_t^4(w_T, \mathcal{L}^4_x)} \lesssim T^{\frac{1}{8}} C^{\frac{b}{4}} \| P_{I_T} f \|_{\mathcal{L}^2_x}
\end{equation}
as the claim will then follow from the almost orthogonality in \eqref{eq:AlmostOrthogonalityL4Strichartz}. It follows from Galilean invariance and lack of oscillation on the time interval $[0,T]$ for frequencies $\lesssim T^{-\frac{1}{2}}$ that
\begin{equation}
\label{eq:AuxL4SEQII}
\| e^{it \partial_x^2} P_N R_C P_I f \|_{L_t^4([0,T],\mathcal{L}^4_x)} \lesssim T^{\frac{1}{4}} \| P_I R_C f \|_{\mathcal{L}^4_x}.
\end{equation}
It remains to obtain an estimate for the number of lattice points $n \in \Z^{\nu}$ with $\langle n \rangle \sim C$ and $\langle n \rangle_\omega \in I_T$. Firstly, divide $I_T$ into intervals $I$ of unit length. We can bound the number of frequencies contained in an interval of unit length by $C^b = C^{\nu-1}$ by the rank of $\Lambda$. Hence, we obtain from the Cauchy-Schwarz inequality
\begin{equation*}
\| P_{I_T} R_C f \|_{\mathcal{L}^\infty_x} \lesssim T^{-\frac{1}{4}} C^{\frac{b}{2}} \| f \|_{\mathcal{L}^2_x}.
\end{equation*}
Interpolation with the trivial $\mathcal{L}^2_x$-estimate gives
\begin{equation*}
\| P_{I_T} R_C f \|_{\mathcal{L}^4_x} \lesssim T^{-\frac{1}{8}} C^{\frac{b}{4}} \| f \|_{\mathcal{L}^2_x},
\end{equation*}
which together with \eqref{eq:AuxL4SEQII} implies \eqref{eq:AuxL4SEQI}.

\smallskip

We are left with proving the almost orthogonality estimate \eqref{eq:AlmostOrthogonalityL4Strichartz}. For $P_{\lesssim 1} R_C e^{it \partial_x^2} f$ this is immediate from the Cauchy-Schwarz inequality. We turn to the contribution of $P_N R_C e^{it \partial_x^2} f$ for $1 \ll N \lesssim C$. We rescale to unit frequencies by $t \to N^2 t$, $x \to Nx$, $\xi \to \xi/N$, and we find
\begin{equation*}
\| P_N R_C e^{it \partial_x^2} f \|^4_{L_t^4([0,T],\mathcal{L}^4_x)} = N^{-2} \| P_1 R_C e^{it \partial_x^2} f' \|^4_{L_t^4([0,N^2 T], \mathcal{L}^4_x)}.
\end{equation*}
We write
\begin{equation*}
\begin{split}
&\quad \| P_1 R_C e^{it \partial_x^2} f' \|^4_{L_t^4([0,N^2 T], \mathcal{L}^4_x)} \\
&= \lim_{L \to \infty} \frac{1}{2L} \int_{[0,N^2 T] \times [-NL,NL]} \big| \sum_{\substack{n \in \Z^{\nu}/N, \\ |N n | \sim C}} e^{i( \langle n \rangle_\omega x + \langle n \rangle^2_{\omega} t)} \hat{f}(N \langle n \rangle_\omega) \big|^4 dx dt.
\end{split}
\end{equation*}
Choose $L \gg N^2 T$, and let $w_A : \R \to \R$ denote a function with compactly supported Fourier transform, which satisfies $|w_A(x)| \gtrsim 1$ for $|x| \lesssim A$ and $\text{supp}(\hat{w}_A) \subseteq B(0,cA^{-1})$. This allows us to dominate:
\begin{equation*}
\begin{split}
&\quad \int_{[0,N^2 T] \times [-NL,NL]} \big| \sum_{\substack{n \in \Z^{\nu}/N, \\ |N n | \sim C}} e^{i( \langle n \rangle_\omega x + \langle n \rangle^2_{\omega} t)} \hat{f}(N \langle n \rangle_\omega) \big|^4 dx dt \\
&\lesssim \int_{\R^2} \underbrace{w_{N^2 T}^4(t) w_{NL}^4(x) \big| \sum_{\substack{n \in \Z^{\nu}/N, \\ |n| \sim C/N}} e^{i( \langle n \rangle_\omega x + \langle n \rangle_\omega^2 t)} \hat{f}(N \langle n \rangle_\omega) \big|^4}_{|F(x,t)|^4} dx dt.
\end{split}
\end{equation*}
Let $\delta = (N^2 T)^{-1}$. The support of the space-time Fourier transform of $F$ is contained in $\mathcal{N}_\delta(\{(\xi, |\xi|^2) : |\xi| \lesssim 1 \})$. Consequently, applying the C\'ordoba--Fefferman square function estimate yields
\begin{equation*}
\int_{\R^2} |F(x,t)|^4 dx dt \lesssim \int \big( \sum_{I_T: N^{-1} T^{-\frac{1}{2}}-\text{interval}} |P_{I_T} F(x,t)|^2 \big)^2 dx dt.
\end{equation*}
We take the exponent $1/4$, apply Minkowski's inequality to interchange the $\ell^2$ with the $L^4$-norm and rescale to find after taking the limit $L\to \infty$
\begin{equation*}
\| P_N R_C e^{it \partial_x^2} f \|_{L_t^4([0,T],\mathcal{L}^4_x)} \lesssim \big( \sum_{I_T: T^{-\frac{1}{2}}-\text{interval}} \| P_{I_T} R_C e^{it \partial_x^2} f \|^2_{L_t^4(w_T,\mathcal{L}^4_x)} \big)^{\frac{1}{2}},
\end{equation*}
which is \eqref{eq:AlmostOrthogonalityL4Strichartz}.

Next, we consider times $T \leq N^{-2}$. There are no Schr\"odinger oscillations, which allows us to estimate
\begin{equation*}
\| e^{it \partial_x^2} P_N R_C f \|_{L_t^4([0,T],\mathcal{L}^4_x)} \lesssim T^{\frac{1}{4}} \| P_N R_C f \|_{\mathcal{L}^4_x}.
\end{equation*}
We divide now the Fourier support $[-2N,-N/2] \cup [N/2,2N]$ into $\sim N$ unit intervals and for every interval of unit length we have like above at most $\lesssim C^{b}$ frequencies of height $C$. Consequently, an application of Bernstein's inequality yields
\begin{equation*}
\| P_N R_C f \|_{\mathcal{L}^4_x} \lesssim (N C^b)^{\frac{1}{4}} \| f \|_{\mathcal{L}^2_x}.
\end{equation*}
Since $T \leq N^{-2}$ taking the preceding displays together, we finish the proof of \eqref{eq:FixedTimeL4StrichartzSEQ}.

\medskip

We turn to the proof of \eqref{eq:FixedTimeL4StrichartzAiry}. For $N \lesssim 1$ we simply use the same frequency counting argument to bound the number of frequencies with height $\sim C$ by $C^{b}$, which gives by H\"older's inequality
\begin{equation*}
\| P_{N} R_C e^{t \partial_x^3} f \|_{L_t^4([0,T],\mathcal{L}^4_x)} \lesssim T^{\frac{1}{4}} C^{\frac{b}{4}} \| P_N R_C f \|_{\mathcal{L}^2_x}.
\end{equation*}

For $N \gg 1$ and $T \geq N^{-3}$ we shall show the almost orthogonal decomposition:
\begin{equation}
\label{eq:CFSquareFunctionEstimateAiryFixedTime}
\| e^{t \partial_x^3} P_N R_C f \|_{L_t^4([0,T],\mathcal{L}^4_x)} \lesssim \big(\sum_{I_T:T^{-\frac{1}{2}} N^{-\frac{1}{2}}-\text{interval}} \| P_N R_C e^{t \partial_x^3} P_{I_T} f \|_{L_t^4(w_T,\mathcal{L}^4_x)}^2 \big)^{\frac{1}{2}}.
\end{equation}
Then the estimate can be concluded like above by counting the frequencies of height $C$, which are contained in an interval of length $N^{-\frac{1}{2}} T^{-\frac{1}{2}}$. To prove \eqref{eq:CFSquareFunctionEstimateAiryFixedTime}, we rescale to unit frequencies to find:
\begin{equation*}
\begin{split}
&\quad \| P_N R_C e^{t \partial_x^3} f \|^4_{L_t^4([0,T],\mathcal{L}^4_x)} \\
 &= N^{-4} \lim_{L \to \infty} \frac{1}{2L} \int_{[0,N^3 T] \times [-NL,NL]} \big| \sum_{\substack{n \in \Z^{\nu}/N, \\ |\langle n \rangle_\omega| \sim 1, \; |n| \sim C}} e^{i(\langle n \rangle_\omega x + t \langle n \rangle_\omega^3)} \hat{f}(N \langle n \rangle_\omega) \big|^4 dx dt \\
&\lesssim N^{-4} \lim_{L \to \infty} \frac{1}{2L} \int_{\R^2} \underbrace{w^4_{N^3 T}(t) w^4_{NL}(x) \big| \sum_{\substack{n \in \Z^{\nu}/N, \\ |\langle n \rangle_\omega| \sim 1, \; |n| \sim C}} e^{i(\langle n \rangle_\omega x + t \langle n \rangle_\omega^3)} \hat{f}(N \langle n \rangle_\omega) \big|^4}_{|F(x,t)|^4} dx dt.
\end{split}
\end{equation*}
Let $\delta = (N^3 T)^{-1}$. Note that $\text{supp}(\hat{F}) \subseteq \mathcal{N}_\delta(\{(\xi,\xi^3): \, |\xi| \sim 1 \})$. We apply the square function estimate recalled in Theorem \ref{thm:CFSquareFunctionEstimateCurvature} to find
\begin{equation*}
\int_{\R^2} |F(x,t)|^4 dx dt \lesssim \int_{\R^2} \big( \sum_{\theta \in \Theta_\delta} |P_\theta F(x,t)|^2 \big)^2 dx dt.
\end{equation*}
Rescaling gives the almost orthogonal decomposition:
\begin{equation*}
\| P_N R_C e^{t \partial_x^3} f \|_{L_t^4([0,T],\mathcal{L}^4_x)} \lesssim \big( \sum_{I_T: T^{-\frac{1}{2}} N^{-\frac{1}{2}}-\text{interval}} \| P_{N} R_C e^{t \partial_x^3} P_{I_T} f \|^2_{L_t^4(w_T,\mathcal{L}^4_x))} \big)^{\frac{1}{2}}.
\end{equation*}
It remains to count the number of frequencies contained in an $T^{-\frac{1}{2}} N^{-\frac{1}{2}}$-frequency interval. Dividing $T^{-\frac{1}{2}} N^{-\frac{1}{2}}$ into unit intervals and using the density parameter, we find that the number of frequencies contained in an interval $I_T$ is estimated by $\langle T^{-\frac{1}{2}} N^{-\frac{1}{2}} \rangle C^b$.  Arguing like above we find
\begin{equation*}
\| P_{I_T} R_C e^{t \partial_x^3} f \|_{L_t^4(w_T,\mathcal{L}^4_x)} \lesssim T^{\frac{1}{4}} \langle T^{-\frac{1}{2}} N^{-\frac{1}{2}} \rangle^{\frac{1}{4}} C^{\frac{b}{4}}  \| P_{I_T} R_C f \|_{\mathcal{L}^2_x}.
\end{equation*}

For $T \leq N^{-3}$ we again find no significant time oscillations and can conclude by H\"older's inequality and counting the frequencies:
\begin{equation*}
\begin{split}
\| e^{t \partial_x^3} P_N R_C u_0 \|_{L_t^4([0,T],\mathcal{L}^4_x(\R))} \lesssim T^{\frac{1}{4}} \| P_N R_C u_0 \|_{\mathcal{L}^4_x(\R)} &\lesssim T^{\frac{1}{4}} (N C^b)^{\frac{1}{4}} \| u_0 \|_{\mathcal{L}^2_x} \\
&\lesssim T^{\frac{1}{8}} N^{-\frac{1}{8}} C^{\frac{b}{4}} \| u_0 \|_{\mathcal{L}^2_x}.
\end{split}
\end{equation*}

The proof is complete.
\end{proof}

%\begin{remark}
%If the diophantine condition holds, then we can divide the $T^{-\frac{1}{2}} N^{-\frac{1}{2}}$-interval directly into $C^{-b}$ subintervals, which improves the estimate to
%\begin{equation*}
%\| P_{I_T} R_C e^{t \partial_x^3} f \|_{L_t^4(w_T,\mathcal{L}^4_x)} \lesssim T^{\frac{1}{4}} \langle T^{-\frac{1}{2}} N^{-\frac{1}{2}} C^b \rangle^{\frac{1}{4}} \| P_{I_T} R_C f \|_{\mathcal{L}^2_x}.
%\end{equation*}
%\end{remark}

\subsection{Strichartz estimates from decoupling}

In the following we suppose that $\Lambda = \omega^{(1)} \cdot \Z^{\nu_1} \times \ldots \times \omega^{(d)} \cdot \Z^{\nu_d}$ for $\omega^{(i)} \in \R^{\nu_i}_{>0}$ non-resonant. Let $b_i=\nu_i-1$, and $b=b_1+\ldots+b_d$.

\smallskip

Now we can formulate the Strichartz estimates from $\ell^2$-decoupling.
\begin{theorem}
\label{thm:FixedTimeStrichartzGeneral}
Let $d \in \N$, $\Lambda \subseteq \R^d$, and $b$ like above, $f \in \mathcal{H}^s_\Lambda$, and $2 < p < \infty$. Then the following estimate holds:
\begin{equation}
\| e^{it \Delta} f \|_{L_t^p([0,1],\mathcal{L}^p_x(\R^d))} \lesssim \| f \|_{\mathcal{H}^s_\Lambda}
\end{equation}
for
\begin{equation*}
s > s^* = b \big( \frac{1}{2} - \frac{1}{p} \big) + \max \big( \frac{d}{2} - \frac{d+2}{p}, 0 \big).
\end{equation*}
\end{theorem}
\begin{proof}
It suffices to prove for $C \in 2^{\N_0}$ with $s$ like above:
\begin{equation*}
\| R_C e^{it \Delta} f \|_{L_t^p([0,1],\mathcal{L}^p_x(\R^d))} \lesssim C^{s} \| f \|_{\mathcal{L}^2_x(\R^d)}.
\end{equation*}

We use Minkowski's inequality to decompose: 
\begin{equation}
\label{eq:AuxDecouplingI}
\| e^{it \Delta} R_C f \|_{L_t^p([0,1],\mathcal{L}^p_x(\R^d))} \leq \sum_{1 \leq N \lesssim C} \| e^{it \Delta} P_N R_C f \|_{L_t^p([0,1],\mathcal{L}^p_x(\R^d))}.
\end{equation}
We apply Theorem \ref{thm:Decoupling}, using the same scaling and approximation argument like in the previous proof with $\delta^{-1} = N^2$. This gives an almost orthogonal decomposition
\begin{equation}
\label{eq:AuxDecouplingII}
\begin{split}
&\quad \| e^{it \Delta} P_N R_C f \|_{L_t^p([0,1],\mathcal{L}^p_x)} \\
&\lesssim_\varepsilon 
\begin{cases}
 &N^\varepsilon \big( \sum_{I: 1-\text{cube}} \| e^{it \Delta} P_N R_C P_I f \|^2_{L_t^p(w_1,\mathcal{L}^p_x)} \big)^{\frac{1}{2}}, \quad \quad \quad \quad \; 2 < p \leq p_d, \\
 &N^{\frac{d}{2} - \frac{d+2}{p} + \varepsilon} \big( \sum_{I: 1-\text{cube}} \| e^{it \Delta} P_N R_C P_I f \|^2_{L_t^p(w_1,\mathcal{L}^p_x)} \big)^{\frac{1}{2}}, \quad p_d \leq p < \infty.
\end{cases}
\end{split}
\end{equation}
After frequency projection to a cube with side-length $1$, the time-evolution is trivialized and it remains to count the frequencies with height $C$ contained in a unit cube. The projection of the cube to a coordinate axis is an interval of unit length, which allows us to count the frequencies contained in the cube by $C^b$.

\smallskip

By the Cauchy-Schwarz inequality we obtain the estimate
\begin{equation*}
\| e^{it \Delta} P_N R_C P_I f \|_{L_t^{\infty}([0,1],\mathcal{L}^{\infty}_x)} \lesssim C^{\frac{b}{2}} \| R_C P_I f \|_{\mathcal{L}^2_x}.
\end{equation*}
Interpolation with the trivial $\mathcal{L}^2$-estimate gives for $p \in (2,\infty)$:
\begin{equation}
\label{eq:AuxDecouplingIII}
\| R_C P_I f \|_{L_t^p([0,1],\mathcal{L}^p_x)} \lesssim C^{b \big( \frac{1}{2} - \frac{1}{p} \big)} \| P_I f \|_{\mathcal{L}^2_x}.
\end{equation}

Taking \eqref{eq:AuxDecouplingI}-\eqref{eq:AuxDecouplingIII} together, we obtain from dyadic summation
\begin{equation*}
\| e^{it \Delta} R_C f \|_{L_t^p([0,1],\mathcal{L}^p_x)} \lesssim_\varepsilon (C^{\frac{d}{2} - \frac{d+2}{p}} \vee 1) C^{b \big( \frac{1}{2} - \frac{1}{p} \big)+\varepsilon} \| R_C f \|_{\mathcal{L}^2_x}.
\end{equation*}
The proof is complete.
\end{proof}

\section{Strichartz estimates upon averaging in time}
\label{section:AveragedStrichartzEstimates}
In this section we show Theorem \ref{thm:AveragedL4Strichartz} and a variant in higher dimensions, see below. Decoupling and square function estimates are again the key ingredients. After recording a variant of Littlewood-Paley decomposition, we can suppose a dyadic frequency localization. Dyadically localized frequencies can be rescaled to unit frequencies, after which decoupling becomes applicable. Since we average in time, we can apply decoupling on arbitrarily small scales. This allows us to perfectly separate the frequencies and show Strichartz estimates with the only derivative loss coming from decoupling.

\subsection{Littlewood-Paley decomposition}

To prove the Strichartz estimates, we use a variant of Littlewood-Paley theory. Let $\chi_1 \in C^\infty_c(-2,2)$ be a radially decreasing function with $\chi_1(\xi) \equiv 1$ for $|\xi| \leq 1$. Define for $N \in 2^{\N}$, $\chi_N: \R^d \to \R_{>0}$, $\chi_N(|\xi|) = \chi(|\xi|/N) - \chi(2 |\xi| / N)$. This gives a smooth partition of unity:
\begin{equation*}
\sum_{N \in 2^{\N_0}} \chi_N(|\xi|) \equiv 1.
\end{equation*}
We define frequency projections via $(P_N f) \widehat(\xi) = \chi_N(|\xi|) \hat{f}(\xi)$. We have the following:
\begin{proposition}[Littlewood-Paley~decomposition]
\label{prop:LittlewoodPaley}
Let $f \in \mathcal{B}_2(\R^d)$. Then it holds for $p \in [2,\infty)$:
\begin{equation*}
\| f \|_{\mathcal{L}^p_x(\R^d)} \lesssim \big( \sum_{N \in 2^{\N_0}} \| P_N f \|^2_{\mathcal{L}^p_x(\R^d)} \big)^{\frac{1}{2}}.
\end{equation*}
\end{proposition}
\begin{proof}
We take suitable Schwartz functions $\varphi_L$, which are rapidly decaying off $[-L,L]^d$ to write:
\begin{equation*}
\| f \|_{\mathcal{L}^p_x(\R^d)}^p = \lim_{L \to \infty} \frac{1}{(2L)^d} \int_{\R^d} |\varphi_L(x) f(x)|^p dx.
\end{equation*}
We require that $\text{supp}(\hat{\varphi}_L) \subseteq B(0,C/L)$ with $|\varphi_L(x)| \gtrsim 1$ for $|x| \leq L$. Now we can apply standard Littlewood-Paley theory in $L^p$-spaces:
\begin{equation*}
\big( \frac{1}{(2L)^d} \int_{\R^d} |\varphi_L(x) f(x)|^p dx \big)^{\frac{1}{p}} \sim \big[ \frac{1}{(2L)^d} \int_{\R^d} \big( \sum_{N \in 2^{\N_0}} |P_N (\varphi_L f)|^2 \big)^{\frac{p}{2}} \big]^{\frac{1}{p}}.
\end{equation*}
We have by Minkowski's inequality:
\begin{equation}
\label{eq:AuxLPEstimate}
\big( \frac{1}{2L} \int_{\R} |\varphi_L(x) f(x)|^p dx \big)^{\frac{1}{p}} \lesssim \big[ \sum_{N \in 2^{\N_0}} \big( \frac{1}{2L} \int_{\R} |P_N (\varphi_L f)|^p \big)^{\frac{2}{p}} \big]^{\frac{1}{2}}.
\end{equation}
Note that by small Fourier support of $\varphi_L$ we have
\begin{equation*}
P_N \varphi_L f = P_N \varphi_L \tilde{P}_N f
\end{equation*}
with $\tilde{P}_N = P_{N/2} + P_N + P_{2N}$, setting $P_{1/2} = 0$. Then by uniform boundedness of $P_N$ in $L^p$ follows from \eqref{eq:AuxLPEstimate}:
\begin{equation*}
\big( \frac{1}{(2L)^d} \int_{\R^d} |\varphi_L(x) f(x)|^p dx \big)^{\frac{1}{p}} \lesssim \big[ \sum_{N \in 2^{\N_0}} \big( \frac{1}{(2L)^d} \int_{\R^d} |\varphi_L(x) \tilde{P}_N f(x)|^p dx \big)^{\frac{2}{p}} \big]^{\frac{1}{2}}.
\end{equation*}
Taking the limit we find by monotone convergence:
\begin{equation*}
\| f \|_{\mathcal{L}^p_x} \lesssim \big( \sum_{N \in 2^{\N_0}} \| \tilde{P}_N f \|_{\mathcal{L}^p_x}^2 \big)^{\frac{1}{2}} \lesssim \big( \sum_{N \in 2^{\N_0}} \| P_N f \|^2_{\mathcal{L}^p_x(\R)} \big)^{\frac{1}{2}}.
\end{equation*}
\end{proof}

\subsection{One-dimensional Strichartz estimates with time-averaging}

Next, we prove Theorem \ref{thm:AveragedL4Strichartz}.
Invoking Proposition \ref{prop:LittlewoodPaley},
%\begin{equation*}
%\| f \|_{\mathcal{L}^2_x(\R^d)} = \big( \sum_{\lambda \in \Lambda} |a_\lambda|^2 \big)^{\frac{1}{2}} \text{ for } f(x) = \sum_{\lambda \in \Lambda} a_\lambda e^{i \lambda x} \in \mathcal{B}_2(\R^d),
%\end{equation*}
it suffices to show frequency-localized estimates:
\begin{align}
\label{eq:FrequencyLocalizedL4SEQ}
\| P_N e^{it \partial_x^2} f \|_{\mathcal{L}^4_{t,x}(\R \times \R)} &\lesssim \| f \|_{\mathcal{L}^2_x(\R)}, \\
\label{eq:FrequencyLocalizedL4Airy}
\| P_N e^{t \partial_x^3} f \|_{\mathcal{L}^4_{t,x}(\R \times \R)} &\lesssim \| f \|_{\mathcal{L}^2_x(\R)}.
\end{align}
By density and limiting arguments, we can suppose that $f$ is a trigonometric polynomial:
\begin{equation*}
f = \sum_{n=1}^M e^{i \lambda_n x} \hat{f}(\lambda_n).
\end{equation*}

\medskip

We turn to the proof of the frequency localized estimates.
The estimates for $N=1$ and $N > 1$ are proven slightly different. For $N=1$, we write
\begin{equation*}
P_1 f = \sum_{ \substack{\lambda \in \Lambda, \\ |\lambda| \leq 1} } e^{i \lambda x} a_\lambda.
\end{equation*}
We consider
\begin{equation*}
e^{it \partial_x^2} P_1 f = \sum_{\substack{\lambda \in \Lambda, \\ |\lambda| \leq 1}} e^{i (\lambda x - t \lambda^2)} a_\lambda.
\end{equation*}
Choose $L$ large enough such that the frequencies are separated by at least $c L^{-\frac{1}{2}}$ (this is clearly possible with the frequencies finite). Then we can write
\begin{equation*}
\int_{[-L,L]^2} |e^{it \partial_x^2} P_1 f |^4 dx dt \lesssim \int_{\R^2} |
\varphi_L(x,t) \cdot e^{it \partial_x^2} P_1 f|^4 dx dt
\end{equation*}
with $|\varphi_L| \gtrsim 1$ on $[-L,L]^2$ and Fourier support contained in $B(0,c^2 L^{-1})$. The space-time Fourier transform of $\varphi_L e^{it \partial_x^2} P_1 f$ is clearly contained in $\mathcal{N}_{c L^{-1}}(\{(\xi,\xi^2) : \xi \in [-1,1] \})$. Let $\delta = L^{-1}$ and let $\Theta_{\delta}$ be a finitely overlapping cover of $\mathcal{N}_{c L^{-1}}(\{(\xi,\xi^2) : \xi \in [-1,1] \})$ with rectangles of size $L^{-\frac{1}{2}} \times L^{-1}$, with long-side pointing into tangential and short side pointing into normal direction. Consequently, we can apply Theorem \ref{thm:CFSquareFunctionEstimateCurvature} to find
\begin{equation*}
\int_{\R^2} |\varphi_L(x,t) e^{it \partial_x^2} P_1 f |^4 dx dt \lesssim \int \big( \sum_{\theta \in \Theta_{\delta}} |P_{\theta}( \varphi_L e^{it \partial_x^2} P_1 f) |^2 \big)^2 dx dt.
\end{equation*}
Exponentiating with $1/4$ and applying Minkowski's inequality we find
\begin{equation*}
\begin{split}
\big( \int_{\R^2} | \varphi_L(x,t) e^{it \partial_x^2} P_1 f |^4 dx dt \big)^{\frac{1}{4}} &\lesssim \big( \int_{\R^2} \big( \sum_{\theta \in \Theta_{\delta}} | \varphi_L e^{it \partial_x^2} P_{\tilde{\theta}} f |^2 \big)^2 dx dt \big)^{\frac{1}{4}} \\
&\lesssim \big[ \sum_{\theta \in \Theta_{\delta}} \big( \int_{\R^2} | \varphi_L e^{it \partial_x^2} P_{\tilde{\theta}} f|^4 dx dt \big)^{\frac{2}{4}} \big]^{\frac{1}{2}}.
\end{split}
\end{equation*}
By the compact Fourier support of $\varphi_L$ we can write $P_{\theta} \varphi_L f = P_{\theta} \varphi_L P_{\tilde{\theta}} f$ with $\tilde{\theta}$ denoting a mild enlargement of $\theta$.

\small

Still, $P_{\tilde{\theta}} f$ contains at most $\mathcal{O}(1)$ frequencies $(\lambda,\lambda^2)$ with $\lambda \in \Lambda_1$ by our choice of $L$. Consequently, the oscillations have been trivialized, and we compute the integral
\begin{equation*}
\int_{\R^2} | \varphi_L e^{it \partial_x^2} P_{\tilde{\theta}} f|^4 dx dt \lesssim L^2 \sum_{\lambda: (\lambda,\lambda^2) \in \tilde{\theta}} |a_\lambda|^4.
\end{equation*}
This gives
\begin{equation*}
\big( \int_{\R^2} | \varphi_L(x,t) e^{it \partial_x^2} P_1 f |^4 dx dt \big)^{\frac{1}{4}} \lesssim L^{\frac{1}{2}} \big( \sum_{\theta \in \Theta_\delta} \big( \sum_{\lambda \in \theta} |a_\lambda|^4 \big)^{\frac{1}{2}} \big)^{\frac{1}{2}} \lesssim L^{\frac{1}{2}} \big( \sum_{\substack{\lambda \in \Lambda, \\ |\lambda| \leq 1}} |a_\lambda|^2 \big)^{\frac{1}{2}}
\end{equation*}
and concludes the proof of \eqref{eq:FrequencyLocalizedL4SEQ} for $N=1$.

\medskip

The estimate
\begin{equation*}
\| P_1 e^{t \partial_x^3} f \|_{\mathcal{L}^4_{t,x}(\R \times \R)} \lesssim \| f \|_{\mathcal{L}^2_x(\R)}
\end{equation*}
is proved along the same lines. The only difference is that we choose $L$ large enough such that frequencies are separated by at least $L^{-\frac{1}{3}}$. Applying Theorem \ref{thm:CFSquareFunctionCubic} with $\delta = L^{-1}$ yields
\begin{equation*}
\int_{\R^2} \big| \varphi_L(x,t) e^{t \partial_x^3} P_1 f \big|^4 dx dt \lesssim \int \big( \sum_{\theta \in \Theta_\delta} \big| P_{\theta} \big( \varphi_L e^{it \partial_x^2} P_1 f \big) \big|^2 \big)^2 dx dt.
\end{equation*}
This again trivializes the time evolution and we find following along the above lines:
\begin{equation*}
\big( \int_{\R^2} |\varphi_L(x,t) e^{t \partial_x^3} P_1 f |^4 dx dt \big)^{\frac{1}{4}} \lesssim L^{\frac{1}{2}} \big( \sum_{\lambda \in \Lambda_1} |a_\lambda|^2 \big)^{\frac{1}{2}}.
\end{equation*}
This concludes \eqref{eq:FrequencyLocalizedL4Airy} for $N=1$.

\medskip

We turn to the estimates \eqref{eq:FrequencyLocalizedL4SEQ} and \eqref{eq:FrequencyLocalizedL4Airy} for $N >1$.
In both cases the proof hinges on Theorem \ref{thm:CFSquareFunctionEstimateCurvature}. We turn to the proof for the Schr\"odinger propagation. Firstly, recall that we assume the frequency set of $P_N f$ to be finite. We denote the minimal separation by $\sigma > 0$. By rescaling we find
\begin{equation}
\label{eq:RescalingSEQ}
\int_{[-L,L] \times [-L,L]} |e^{it \partial_x^2} P_N f|^4 dx dt = N^{-3} \int_{[-N^2 L, N^2 L] \times [-N L, NL]} |e^{it \partial_x^2} P_1 f'|^4 dx dt.
\end{equation}
We dominate with suitable Schwartz functions with compact Fourier support:
\begin{equation*}
\int_{\R^2} w_{N^2 L}^4(t) w^4_{NL}(x) |e^{it \partial_x^2} P_1 f'|^4 dx dt = \int_{\R^2} |\underbrace{w_{N^2 L}(t) w_{NL}(x) e^{it \partial_x^2} P_1 f'}_{F(x,t)}|^4 dx dt.
\end{equation*}
We have $\text{supp}(\hat{F}) \subseteq \mathcal{N}_{(NL)^{-1}}( \{(\xi,\xi^2) : \frac{1}{2} \leq |\xi| \leq 2 \})$. The frequencies are separated by $N^{-1} \sigma$, with $\sigma$ denoting the original separation of frequencies of $P_N f$. We aim to apply Theorem \ref{thm:CFSquareFunctionEstimateCurvature} with $\delta = (NL)^{-1}$. Choose $L$ large enough such that $(NL)^{-\frac{1}{2}} \ll N^{-1} \sigma$. Then applying Theorem \ref{thm:CFSquareFunctionEstimateCurvature} gives
\begin{equation*}
\int_{\R^2} |F(x,t)|^4 dx dt \lesssim \int_{\R^2} \big( \sum_{\theta \in \Theta_\delta} |F_{\theta}(x,t)|^2 \big)^2 dx dt.
\end{equation*}
With the frequencies being separated, i.e., only finitely many $(\lambda,\lambda^2)$ are contributing to $F_\theta$, we can conclude the argument like in the proof for $N=1$.

\smallskip

For the Airy evolution only the rescaling changes:
\begin{equation*}
\int_{[-L,L] \times [-L,L]} |P_N e^{t \partial_x^3} f |^4 dx dt = N^{-4} \int_{[-N^3 L, N^3 L] \times [-NL,NL]} |P_1 e^{t \partial_x^3} f'|^4 dx dt.
\end{equation*}
Dominating like above
\begin{equation*}
\int_{[-N^3 L, N^3 L] \times [-NL,NL]} |P_1 e^{t \partial_x^3} f'|^4 dx dt \lesssim \int_{\R^2} | \underbrace{w_{N^3 L}(t) w_{NL}(x) P_1 e^{t \partial_x^3} f'}_{F}|^4 dx dt.
\end{equation*}
Clearly, $\text{supp}(\hat{F}) \subseteq \mathcal{N}_{(NL)^{-1}}(\{(\xi,\xi^3) : |\xi| \sim 1 \})$. Since the curve $\{(\xi, \xi^3) : |\xi| \sim 1 \}$ has curvature, we can again apply Theorem \ref{thm:CFSquareFunctionEstimateCurvature} and conclude the argument like above.

$\hfill \Box$

\subsection{Strichartz estimates for quasi-periodic functions in higher dimensions}

We record a higher-dimensional version of Theorem \ref{thm:AveragedL4Strichartz}. Let $\Lambda = \omega^{(1)} \cdot \Z^{\nu_1} \times \ldots \times \omega^{(d)} \cdot \Z^{\nu_d} \subseteq \R^d$ with $\omega^{(i)}$ non-resonant. 
%We have $\Lambda = \sum_{j \in \N, (*)} \omega_j \Z^d$ (by $(*)$ we indicate that only finitely many summands are non-vanishing) for $(\omega_j)_{j \in \N} \subseteq \R$, $|\omega_j| \leq 2$, linearly independent over $\Q$, i.e., for any $\{j_1,\ldots,j_N\} \subseteq \N$ and $\lambda_1,\ldots,\lambda_N \in \Q$ it holds
%\begin{equation*}
%\sum_{i=1}^N \lambda_i \omega_{j_i} = 0 \Rightarrow \lambda_1 = \ldots = \lambda_N = 0.
%\end{equation*}
%Moreover, any $\lambda \in \Lambda$ has the unique representation:
%\begin{equation*}
%\lambda = \sum_{i=1}^N \omega_{j_i} n_i.
%\end{equation*}
%We have the Fourier series expansion
%\begin{equation*}
%f(x) = \sum_{\underline{n} \in (\Z^d)^{\N}} e^{i \langle \underline{n} \rangle_{\underline{\omega}} x} \hat{f}(\langle \underline{n} \rangle_{\underline{\omega}}), \quad \langle \underline{n} \rangle_{\underline{\omega}} = \sum_{j \in \N, (*)} \omega_j n_j.
%\end{equation*}
%and define Sobolev spaces by
%\begin{equation*}
%\| f \|^2_{\mathcal{H}^s_\Lambda} = \sum_{\underline{n} \in (\Z^d)^{\N}} \langle \underline{n} \rangle^{2s} |\hat{f}(\langle \underline{n} \rangle_{\underline{\omega}})|^2.
%\end{equation*}
%We impose an upper bound on the modulus of $\omega_j$ to estimate the frequency $\langle \underline{n} \rangle_{\underline{\omega}}$ in terms of the height $|\underline{n}|$.
Moreover, we suppose the following diophantine condition:

\smallskip

\emph{There are $\alpha_i,\beta_i>0$ such that for any $n_i \in \Z^{\nu_i} \backslash 0$ it holds}
\begin{equation}
\label{eq:DiophantineConditionInfDim}
|\langle n_i \rangle_{\underline{\omega}} | \geq \alpha_i |n_i|^{-\beta_i}.
\end{equation}
Denote $\beta = \beta_1+\ldots+\beta_d$.

\smallskip

We show the following:
\begin{theorem}
\label{thm:AveragedStrichartzHigh}
Let $\Lambda$ be like above, $p_d = \frac{2(d+2)}{d}$, and $\varepsilon > 0$. For $f \in \mathcal{H}^\varepsilon_\Lambda$ the following estimate holds:
\begin{equation*}
\| e^{it \Delta} f \|_{\mathcal{L}^{p_d}_{t,x}(\R \times \R^d)} \lesssim \| f \|_{\mathcal{H}^\varepsilon_\Lambda}.
\end{equation*}
\end{theorem}
\begin{proof}
Let $p=p_d$ for brevity. We can use a Littlewood-Paley decomposition provided by Proposition \ref{prop:LittlewoodPaley} to reduce to
\begin{equation*}
\| P_N e^{it \Delta} f \|_{\mathcal{L}^p_{t,x}(\R \times \R^d)} \lesssim  \| f \|_{\mathcal{H}^\varepsilon_\Lambda}.
\end{equation*}
For the above, it suffices to show for $C \gtrsim N$:
\begin{equation}
\label{eq:AuxDecouplingHighIV}
\| P_N R_C e^{it \Delta} f \|_{\mathcal{L}^p_{t,x}(\R \times \R^d)} \lesssim_\varepsilon (NC)^\varepsilon \| f \|_{\mathcal{L}^2_x}.
\end{equation}
We can estimate the gap between two distinct frequencies $\underline{n} \neq \underline{n}' \in \Z^{\nu}$ of height $C$ by the diophantine condition:
\begin{equation*}
|\langle \underline{n} \rangle_{\underline{\omega}} - \langle \underline{n} \rangle_{\underline{\omega}} | \gtrsim C^{-\beta}.
\end{equation*}
We can choose $L$ large enough such that $(NL)^{-\frac{1}{2}} \ll C^{-\beta}$. Then, it holds
\begin{equation}
\label{eq:AuxDecouplingHighI}
\begin{split}
&\quad \int_{[-L,L] \times [-L,L]^d} \big| \sum_{\substack{ \lambda \in \Lambda, \\ |\lambda| \sim N, \\ h(\lambda) \sim C }} e^{ i (\lambda x - t \lambda^2)} \hat{f}(\lambda) \big|^p dx dt \\
&= N^{-(d+2)} \int_{[-N^2 L, N^2 L] \times [-N L, NL]} \big| \sum_{\substack{ \lambda' \in \Lambda/N, \\ |\lambda'| \sim 1, \\ h(N \lambda') \sim C }} e^{i (\lambda' x'  - t' (\lambda')^)} \hat{f}(N \lambda') \big|^p dx dt \\
 &\lesssim N^{-(d+2)} \int_{\R^{d+1}} \big| \underbrace{w_{N^2 L}(t) w_{NL}(x) \sum_{\substack{ \lambda' \in \Lambda/N, \\ |\lambda'| \sim 1, \\ h(N \lambda') \sim C }} e^{i (\lambda' x'  - t' (\lambda')^2)} \hat{f}(N \lambda')}_{F(x,t)} \big|^p dx dt.
\end{split}
\end{equation}

We have that $\text{supp}(\hat{F}) \subseteq \mathcal{N}_{\delta}(\{(\xi,-|\xi|^2):|\xi| \leq 1\})$ for $\delta = (NL)^{-1}$. Decoupling gives a decomposition into caps on the Fourier side of size $\delta^{\frac{1}{2}}$ in the tangential direction and $\delta$ in the normal direction:
\begin{equation}
\label{eq:AuxDecouplingHighII}
\| F \|_{L^p(\R^{d+1})} \lesssim_\varepsilon (N C)^\varepsilon \big( \sum_{\theta \in \Theta_\delta} \| P_{\theta} F \|^2_{L^p(\R^{d+1})} \big)^{\frac{1}{2}}.
\end{equation}
Like in the previous arguments, we see that there are only finitely many frequencies $\lambda \in \Lambda$ contributing to $P_{\theta} F$. Hence, we can carry out the integration, average, and obtain from \eqref{eq:AuxDecouplingHighI}-\eqref{eq:AuxDecouplingHighII} the estimate \eqref{eq:AuxDecouplingHighIV}. The proof is complete.

\end{proof}

\section{Low regularity local well-posedness for the nonlinear Schr\"odinger equation with quasi-periodic initial data}
\label{section:LWP}

This section is devoted to the proof of Theorem \ref{thm:LWPNLS}. We begin with the introduction of adapted function spaces:

\subsection{Adapted function spaces}
We define adapted function spaces behaving well with sharp-time cutoff. These constitute a logarithmic refinement of Fourier restriction spaces used by Bourgain \cite{Bourgain1993A,Bourgain1993B} in the analysis of the periodic case. We choose these function spaces because the key bilinear estimates are readily formulated within.

\smallskip

 Functions with bounded variation were investigated by Wiener \cite{Wiener1979} for the first time. In the following we shall also need the predual $U^p$-spaces. We refer to the detailed exposition of Hadac--Herr--Koch \cite{HadacHerrKoch2009} (and \cite{HadacHerrKoch2009Erratum}) for details without which the following account might appear a little bit sketchy.
 
 \begin{definition}
 Let $\mathcal{Z}$ be the collection of finite non-decreasing sequences $\{t_k\}_{k=0}^K$ in $\R$. For $1 \leq p < \infty$ we define $V^p$ as the space of all right-continuous functions $u:\R \to \C$ with $\lim_{t \to -\infty} u(t)=0$ and
 \begin{equation*}
 \| u \|_{V^p} = \sup_{\{t_k\}_{k=0}^K \in \mathcal{Z}} \big( \sum_{k=1}^K |u(t_k)-u(t_{k-1})|^p \big)^{\frac{1}{p}} < \infty.
 \end{equation*}
 \end{definition}

The following $U^p_{\Delta}$-/$V^p_{\Delta}$-spaces will be useful for intermediate steps. The function space, in which we formulate the well-posedness results, are the $Y^s_{\pm}$-spaces defined below.
\begin{definition}
Let $1 \leq p < \infty$. We define
\begin{equation*}
V^p_{\Delta} = \{ u: \R \times \R^d \to \C : e^{-it \Delta} u \in V^p \mathcal{L}^2(\R^d) \}
\end{equation*}
with norm given by
\begin{equation*}
\| u \|_{V^p_{\Delta}} = \big( \sup_{\{t_k\}_{k=0}^K \in \mathcal{Z}} \sum_{k=1}^K \| e^{-it_k \Delta} u(t_k) - e^{-it_{k-1} \Delta} u(t_{k-1}) \|^p_{\mathcal{L}^2_x} \big)^{\frac{1}{p}} < \infty.
\end{equation*}
\end{definition}

The predual $U^p$-spaces admit an atomic decomposition (see \cite{HadacHerrKoch2009}). As a consequence we obtain the transfer principle:
Suppose that for $2 \leq p < \infty$ the estimate
\begin{equation*}
\| R_M e^{it \Delta} f \|_{L_t^p([0,1],\mathcal{L}^p_x(\R))} \lesssim C(M) \| f \|_{\mathcal{L}^2_x}
\end{equation*}
holds. Then we have the following consequence:
\begin{equation*}
\| R_M u \|_{L_t^p([0,1],\mathcal{L}^p_x(\R))} \lesssim C(M) \| u \|_{U^p_\Delta}.
\end{equation*}

In the following fix $\Lambda \subseteq \R^d$ like in \eqref{eq:HigherDimensionalLattice}.
\begin{definition}
For $s \in \R$, we define $Y^s$ as the space of functions $u: \R \times \R^d \to \C$ with $\text{supp} (\hat{u}(t)) \subseteq \Lambda$ such that $e^{it \langle n \rangle_\omega^2} \widehat{u(t)}(\langle n \rangle_\omega)$ lies in $V^2$ for any $n \in \Z^{\nu_1} \times \ldots \times \Z^{\nu_d} =: \Z^{\nu}$ and the norm is defined as
\begin{equation}
\| u \|^2_{Y^s} = \sum_{n \in \Z^{\nu}} (1+|n|^2)^s \| e^{i t \langle n \rangle_\omega^2} \widehat{u(t)}(\langle n \rangle_\omega) \|^2_{V^2} < \infty.
\end{equation}
For $T>0$ and $u:[0,T] \times \R^d \to \C$ we define
\begin{equation*}
\| u \|^2_{Y^s_{T}} = \sum_{n \in \Z^{\nu}} (1+|n|^2)^s \| e^{ it \langle n \rangle_\omega^2} \widehat{u(t)}(\langle n \rangle_\omega) \|^2_{V^2(0,T)}.
\end{equation*}
\end{definition}

We have the following embeddings:
\begin{equation*}
Y^0_T \hookrightarrow V^2_\Delta \hookrightarrow U^p_\Delta \text{ for } p > 2.
\end{equation*}
The first is immediate from the definition and for the second we refer to \cite{HadacHerrKoch2009}.

We summarize further properties to be used freely in the following:
\begin{proposition}[{\cite[Section~2]{HerrTataruTzvetkov2011},~\cite{Koch2014}}]
The following holds:
\begin{itemize}
\item Let $A,B \subseteq \R^d$ be disjoint. For $s \geq 0$ we have
\begin{equation*}
\| P_{A \cup B} u \|^2_{Y^s_T} = \|P_A u \|^2_{Y^s} + \| P_B u \|^2_{Y^s}.
\end{equation*}
\item For $s \geq 0$, $T>0$, and $f \in L_T^1 \mathcal{H}^s_\Lambda$, denote
\begin{equation*}
\mathcal{I}(f)(t) = \int_0^t e^{i(t-s) \Delta} f(s) ds
\end{equation*}
we have with $\chi_T = \chi_{[0,T)}$ denoting the indicator function:
\begin{equation*}
\| \chi_{T} \cdot \mathcal{I}(f) \|_{Y^s} \lesssim \sup_{\substack{v \in Y^{-s}: \\ \| v \|_{Y^{-s}} \leq 1 }} \big| \int_0^T \langle f, \bar{v} \rangle_{\mathcal{L}^2_x} dt \big|.
\end{equation*}
\item For $T>0$ and a function $\phi \in \mathcal{H}^s_\Lambda$ we have
\begin{equation*}
\| \chi_T e^{it \Delta} \phi \|_{Y^s} \lesssim \| \phi \|_{\mathcal{H}^s_\Lambda}
\end{equation*}
and for $u \in Y^s_T$:
\begin{equation*}
\| u \|_{L^\infty_T \mathcal{H}^s_\Lambda} \lesssim \| u \|_{Y^s_T}.
\end{equation*}
\end{itemize}
\end{proposition}

\subsection{Multilinear Strichartz estimates in adapted function spaces}

In the following we denote with $Q^{a}_C$ for $a \in \Z^{\nu}$ the following frequency projection:
\begin{equation*}
Q^a_C f = \sum_{\substack{n \in \Z^{\nu}, \\ |n-a| \leq C}} e^{i \langle n \rangle_\omega x} \hat{f}(\langle n \rangle_\omega).
\end{equation*}

We record the following consequence of Galilean invariance:
\begin{proposition}[Galilean~invariance]
\label{prop:GalileanInvariance}
Let $2 \leq p \leq \infty$, $T \in (0,1]$. Suppose that the following estimate holds:
\begin{equation}
\label{eq:GalileanInvarianceI}
\| e^{it \Delta} R_C f \|_{L_t^p([0,T],\mathcal{L}^p_x(\R^d))} \lesssim T^\delta C^s \| R_C f \|_{\mathcal{L}^2_x(\R^d)}.
\end{equation}
Then it holds with implicit constant independent of $a \in \Z^{\nu}$:
\begin{equation}
\label{eq:GalileanInvarianceII}
\| e^{it \Delta} Q_C^a f \|_{L_t^p([0,1],\mathcal{L}^p_x(\R^d))} \lesssim T^\delta C^s \| Q_C^a f \|_{\mathcal{L}^2_x(\R^d)}.
\end{equation}
\end{proposition}
\begin{proof}
We write
\begin{equation*}
\begin{split}
e^{it \Delta} Q_C^a f &= \sum_{\substack{n \in \Z^{\nu}, \\ |n-a| \leq C}} e^{i( \langle n \rangle_\omega x- t \langle n \rangle_\omega^2)} \hat{f}(\langle n \rangle_\omega) \\
&= \sum_{\substack{n' \in \Z^{\nu}, \\ |n'| \leq C}} e^{i((\langle n' \rangle_\omega + \langle a \rangle_\omega) x - t(\langle n' \rangle_\omega + \langle a \rangle_\omega)^2)} \hat{f}(\langle n' \rangle_\omega + \langle a \rangle_\omega) \\
&= e^{i \langle a \rangle_\omega x - i t \langle a \rangle^2_{\omega}} \sum_{\substack{n' \in \Z^{\nu}, \\ |n'| \leq C}} e^{i( \langle n' \rangle_\omega (x - 2t \langle a \rangle_\omega) - t \langle n' \rangle_\omega^2)} \hat{f}(\langle n' + a \rangle_\omega).
\end{split}
\end{equation*}
Consequently, after the linear change of variables $x \to x - 2t \langle a \rangle_\omega$ for fixed $t \in [0,T]$, which leaves the $\mathcal{L}^p_x$-norm invariant, \eqref{eq:GalileanInvarianceI} becomes applicable. This concludes the proof of \eqref{eq:GalileanInvarianceII}.
\end{proof}

% to $\langle n \rangle_\omega$ for $n \in \Z^2$ with $|n_1-a| \leq C$ and $|n_2-b| \leq C$. 
For sake of illustration, we record the following one-dimensional instance, which is a
consequence of Theorem \ref{thm:L4StrichartzFixedTime}, the above, and the transfer principle:
\begin{corollary}
\label{cor:L4StrichartzTransfer}
With $\Lambda = \omega \cdot \Z^{\nu} \subseteq \R$ and notations like above, the following estimate holds:
\begin{equation*}
\| Q^a_C e^{it \partial_x^2} f \|_{L_t^4([0,T],\mathcal{L}^4_x)} \lesssim_\varepsilon T^{\frac{1}{8}} C^{\frac{b}{4}+\varepsilon} \| Q_C f \|_{Y_T^0}.
\end{equation*}
\end{corollary}
We can now formulate the bilinear refinement, which will be crucial for the local well-posedness result for the cubic NLS on $\R$:
\begin{proposition}
\label{prop:BilinearStrichartzEstimateSEQ}
Let $\Lambda = \omega \cdot \Z^{\nu} \subseteq \R$ be like above, and $C_1,C_2 \in 2^{\N_0}$ with $C_1 \leq C_2$. Then the following estimate holds:
\begin{equation*}
\| R_{C_1} u_1 R_{C_2} u_2 \|_{L_t^2([0,T],\mathcal{L}^2_x)} \lesssim_\varepsilon T^{\frac{1}{4}} C_1^{\frac{b}{2}+\varepsilon} \| R_{C_1} u_1 \|_{Y_T^0} \| R_{C_2} u_2 \|_{Y_T^0}.
\end{equation*}
\end{proposition}
\begin{proof}
We use almost orthogonality induced by the convolution constraint to write
\begin{equation*}
\| R_{C_1} u_1 R_{C_2} u_2 \|^2_{L_t^2([0,T],\mathcal{L}^2_x)} \lesssim \sum_{a: C_1-\text{separated}} \| R_{C_1} u_1 Q^a_{C_1} R_{C_2} u_2 \|_{L_t^2([0,T],\mathcal{L}^2_x)}^2
\end{equation*}
with $\sum_{a: C_1 - \text{separated}} Q_{C_1}^a R_{C_2} f = R_{C_2} f$. Then we can apply the Strichartz estimate from Corollary \ref{cor:L4StrichartzTransfer} twice and assemble the frequency projection by the properties of the function spaces:
\begin{equation*}
\begin{split}
\sum_a \| R_{C_1} u_1 Q^a_{C_1} R_{C_2} u_2 \|^2_{L_t^2([0,T],\mathcal{L}^2_x)} &\leq \| R_{C_1} u_1 \|^2_{L_t^4([0,T],\mathcal{L}^4_x)} \sum_a \| Q^a_{C_1} R_{C_2} u_2 \|^2_{L_t^4([0,T],\mathcal{L}^4_x)} \\
&\lesssim_\varepsilon T^{\frac{1}{2}} C_1^{b+\varepsilon} \| R_{C_1} u_1 \|^2_{Y_T^0} \| R_{C_2} u_2 \|^2_{Y_T^0}.
\end{split}
\end{equation*}
\end{proof}

The multilinear generalization reads as follows:
\begin{proposition}
\label{prop:MultilinearEstimate}
Let $d \geq 1$, and $\Lambda = \omega^{(1)} \cdot \Z^{\nu_1} \times \ldots \times \omega^{(d)} \cdot \Z^{\nu_d} \subseteq \R^d$, with $b_i = \nu_i - 1$, and let $b = b_1+\ldots+b_d$. Let $m \geq 3$ for $d=1$ and $m \geq 2$ for $d \geq 2$, and $M_1 \geq M_2 \geq \ldots \geq M_m$. Let $s>b \big( \frac{1}{2} - \frac{1}{2m} \big) + \frac{d}{2} - \frac{d+2}{2m}$. Then the following estimate holds for some $\varepsilon > 0$, $\delta>0$:
\begin{equation}
\label{eq:MultilinearEstimate}
\| R_{M_1} u_1 \ldots R_{M_m} u_m \|_{L_t^2([0,T],\mathcal{L}^2_x(\R^d))} \lesssim T^\delta M_2^{2s} (M_3 \ldots M_m)^{s} \prod_{i=1}^m \| R_{M_i} u_i \|_{Y_T^0}.
\end{equation}
\end{proposition}
\begin{proof}
By convolution constraint and almost orthogonality, we can decompose the frequencies of maximal height $M_1$ into frequencies of height $M_2$ up to Galilean invariance. We write
\begin{equation*}
\begin{split}
&\quad \| R_{M_1} u_1 \ldots R_{M_m} u_m \|_{L_t^2([0,T],\mathcal{L}^2_x(\R^d))}^2 \\
&\lesssim \sum_{a:M_2-\text{separated}} \| Q_{M_2}^a R_{M_1} u_1 R_{M_2} u_2 \ldots R_{M_m} u_m \|^2_{L_t^2([0,T],\mathcal{L}^2_x(\R^d))}.
\end{split}
\end{equation*}
Now we apply H\"older's inequality
\begin{equation*}
\begin{split}
&\quad \| Q_{M_2}^a R_{M_1} u_1 \ldots R_{M_m} u_m \|_{L_t^2([0,T],\mathcal{L}^2_x(\R^d))} \\
&\leq \| Q_{M_2}^a R_{M_1} u_1 \|_{L_t^{2m}([0,T],\mathcal{L}^{2m}_x} \prod_{i=2}^m \| R_{M_i} u_i \|_{L_t^{2m}([0,T],\mathcal{L}^{2m}_x(\R^d))}
\end{split}
\end{equation*}
and next Theorem \ref{thm:FixedTimeStrichartzGeneral} together with Proposition \ref{prop:GalileanInvariance} and the transfer principle: This allows us to estimate the above by
\begin{equation}
\label{eq:AuxMultiEstimateI}
\begin{split}
&\quad \| Q_{M_2}^a R_{M_1} u_1 \|_{L_t^{2m}([0,T],\mathcal{L}^{2m}_x} \prod_{i=2}^m \| R_{M_i} u_i \|_{L_t^{2m}([0,T],\mathcal{L}^{2m}_x(\R^d))} \\
&\lesssim M_2^{2s} (M_3 \ldots M_m)^s \| Q_{M_2}^a R_{M_1} u_1 \|_{Y^0_T} \prod_{i=2}^m \| R_{M_i} u_i \|_{Y^0_T}.
\end{split}
\end{equation}
We interpolate with the following estimate, which is a consequence of H\"older's inequality:
\begin{equation}
\label{eq:AuxMultiEstimateII}
\begin{split}
&\quad \| R_{M_1} u_1 \ldots R_{M_m} u_m \|_{L_t^2([0,T],\mathcal{L}^2_x(\R^d))} \\
&\leq \| R_{M_1} u_1 \|_{L_t^2([0,T],\mathcal{L}^2_x(\R^d))} \prod_{i=2}^d \| R_{M_i} u_i \|_{L_t^{\infty}([0,T],\mathcal{L}^\infty_x(\R^d))} \\
&\leq T^{\frac{1}{2}} (M_2 \ldots M_d)^{C} \prod_{i=1}^m \| R_{M_i} u_i \|_{Y_T^0}.
\end{split}
\end{equation}
Above $C=\nu_1+\ldots+\nu_d$.
Interpolation of \eqref{eq:AuxMultiEstimateI} and \eqref{eq:AuxMultiEstimateII} yields that there for any $\varepsilon>0$ there is $\delta>0$ such that
\begin{equation*}
\| R_{M_1} u_1 \ldots R_{M_m} u_m \|_{L_t^2([0,T],\mathcal{L}^2_x(\R^d))} \lesssim T^\delta M_2^{2s+\varepsilon} (M_3 \ldots M_m)^{s+\varepsilon} \prod_{i=1}^m \| R_{M_i} u_i \|_{Y_T^0}.
\end{equation*}
Redefining $s \to s+\varepsilon$ finishes the proof.
\end{proof}
\begin{remark}
The trilinear estimate for $d=1$ reads
\begin{equation*}
\| R_{M_1} u_1 R_{M_2} u_2 R_{M_3} u_3 \|_{L_t^2([0,T],\mathcal{L}^2_x(\R^d))} \lesssim T^\delta M_2^{2s} M_3^s \prod_{i=1}^3 \| R_{M_i} u_i \|_{Y_T^0}
\end{equation*}
for $s>\frac{b}{3}$. We remark that an alternative estimate is provided by
\begin{equation*}
\| R_{M_1} u_1 R_{M_2} u_2 R_{M_3} u_3 \|_{L_t^2([0,T],\mathcal{L}^2_x(\R^d))} \lesssim T^{\frac{1}{4}} M_2^{2r} M_3^{\frac{\nu}{2}} \prod_{i=1}^3 \| R_{M_i} u_i \|_{Y_T^0}
\end{equation*}
for $r > \frac{b}{4}$. If $M_3 \ll M_2$, then this estimate is more favorable than the one recorded above and will also give an improvement of the local well-posedness result for the quintic NLS. We omit the details since it does not appear to provide a sharp estimate.
\end{remark}

\subsection{Proof of local well-posedness}

Recall that $\Lambda = \omega \cdot \Z^{\nu} \subseteq \R$ with $\omega_i \in \R_{>0}$ linearly independent over $\Q$, and $b = \nu - 1$. We consider the Cauchy problem:
\begin{equation}
\label{eq:CubicNLSDetailed}
\left\{ \begin{array}{cl}
i \partial_t u + \partial_{x}^2 u &= \pm |u|^2 u, \quad (t,x) \in \R \times \R, \\
u(0) &= u_0 \in \mathcal{H}^s_{\Lambda}(\R).
\end{array} \right.
\end{equation}

We formulate a detailed version of Theorem \ref{thm:LWPNLS}:
\begin{theorem}[Local~well-posedness~of~NLS]
\label{thm:LWPNLSDetailed}
Let $s > \frac{b}{2}$. Then there is a real-analytic data-to-solution mapping for \eqref{eq:CubicNLSDetailed}
\begin{equation*}
S_T^s : \mathcal{H}^s_{\Lambda} \to Y^s_T \subseteq C([0,T],\mathcal{H}^s_{\Lambda}), \quad u_0 \mapsto u,
\end{equation*}
for some $T=T(s,\| u_0 \|_{\mathcal{H}^s_\Lambda})$, which is lower semi-continuous, decreasing in the second variable, and satisfies  $\liminf_{a \to 0} T(s,a) \gtrsim 1$.
\end{theorem}

\begin{proof}[Proof~of~Theorem~\ref{thm:LWPNLSDetailed}]
We obtain local-in-time solutions as fixed points of the integral equation for $s>\frac{b}{2}$ given by
\begin{equation*}
\Phi_{u_0}: Y_T^s \to Y_T^s, \quad u(t) =  \chi_T(t) e^{it \partial_x^2} u_0 - i \chi_T(t) \int_0^t e^{i(t-s) \partial_x^2} (|u|^2 u )(s) ds.
\end{equation*}
$T$ will be chosen as $T=T(\| u_0 \|_{\mathcal{H}_{\Lambda}^s})$. Details regarding the application of the contraction mapping principle can be found in the exposition on \emph{quantitative well-posedness} by Bejenaru--Tao \cite{BejenaruTao2006}.

\smallskip

Note that for $u_0 \in \mathcal{H}^s_{\Lambda}$ we have
\begin{equation*}
\| e^{it \partial_x^2} u_0 \|_{L_T^\infty \mathcal{H}^s_x} \lesssim \|u_0 \|_{\mathcal{H}^s_x}.
\end{equation*}
So it remains to establish a trilinear estimate (cf. \cite{BejenaruTao2006}), which is the content of the following proposition:
\begin{proposition}
\label{prop:TrilinearEstimateNLS}
Let $s > \frac{b}{2}$. Then the following trilinear estimate holds for $s \geq s' > \frac{b}{2}$, and $T \in (0,1]$:
\begin{equation}
\label{eq:TrilinearEstimateNLS}
\big\|  \chi_T(t) \int_0^t e^{i(t-s) \partial_x^2} (u_1 \overline{u_2} u_3)(s) ds \big\|_{Y_T^s} \lesssim T^{\frac{1}{4}} \| u_1 \|_{Y^s_T} \| u_2 \|_{Y^{s'}_T} \| u_3 \|_{Y^{s'}_T}.
\end{equation}
\end{proposition}
\begin{proof}
To this end, we use duality to write
\begin{equation*}
\big\|  \chi_T(t) R_C \int_0^t e^{i(t-s) \partial_x^2} (u_1 \overline{u_2} u_3)(s) ds \big\|_{Y_T^s} \lesssim C^s \sup_{\| v \|_{Y_T^0} = 0} \iint \overline{R_C v} R_C (u_1 \overline{u_2} u_3).
\end{equation*}
In the following we omit the complex conjugation, as it can be readily checked that the estimates do not depend on the conjugation. This follows from the invariance of the $L^4$-norm under complex conjugation.

We carry out a height frequency decomposition of $R_C(u_1 u_2 u_3)$: Write
\begin{equation*}
R_C(u_1 u_2 u_3) = \sum_{C_1,C_2,C_3} R_C( R_{C_1} u_1 R_{C_2} u_2 R_{C_3} u_3),
\end{equation*}
and it is easy to see that $C_{\max} = \max(C_1,C_2,C_3) \gtrsim C$ by otherwise impossible frequency interaction.

\medskip

We analyze two cases:\\
$(I)$ $C_{\max} \sim C$ and $(II)$ $C_{\max} \gg C$.

In case $(I)$ we suppose that $C_1 = C_{\max}$. Applying Proposition \ref{prop:BilinearStrichartzEstimateSEQ} twice yields
\begin{equation*}
\begin{split}
\iint R_C v R_{C_1} u_1 R_{C_2} u_2 R_{C_3} u_3 dx dt &\leq \| R_C v R_{C_2} u_2 \|_{L_t^2([0,T],\mathcal{L}^2_x)} \| R_{C_1} u_1 R_{C_3} u_3 \|_{L_t^2([0,T],\mathcal{L}^2_x)} \\
&\lesssim_\varepsilon T^{\frac{1}{4}} C_2^{\frac{1}{2}+\varepsilon} C_3^{\frac{1}{2}+\varepsilon} \| R_C v \|_{Y_T^0} \prod_{i=1}^3 \| R_{C_i} u_i \|_{Y^0_T}.
\end{split}
\end{equation*}
At this point \eqref{eq:TrilinearEstimateNLS} follows from dyadic summation in $C_2, C_3$ and square summation in $C \sim C_1$.

\smallskip

$(II)$ In this case we have $C_{\text{med}} = \max( \{C_1,C_2,C_3 \} \backslash C_{\max}) \sim C_{\max}$ by otherwise impossible frequency interaction. Suppose for definiteness that $C_1 \sim C_2 \sim C_{\max}$. Applying two bilinear Strichartz estimates from Proposition \ref{prop:BilinearStrichartzEstimateSEQ} gives
\begin{equation*}
\begin{split}
\iint R_C v R_{C_1} u_1 R_{C_2} u_2 R_{C_3} u_3 dx dt &\leq \| R_C v R_{C_1} u_1 \|_{L_t^2([0,T],\mathcal{L}^2_x)} \| R_{C_2} u_2 R_{C_3} u_3 \|_{L_t^2([0,T],\mathcal{L}^2_x)} \\
&\lesssim_\varepsilon T^{\frac{1}{4}} C^{\frac{1}{2}+\varepsilon} C_3^{\frac{1}{2}+\varepsilon} \| R_C v \|_{Y_T^0} \prod_{i=1}^3 \| R_{C_i} u_i \|_{Y^0_T}.
\end{split}
\end{equation*}
Summing in $C$, $C_3$ and square summation in $C_1 \sim C_2$ yields the claim.
\end{proof}
With the trilinear estimate at hand, quantitative well-posedness as described by Be\-je\-naru-\--Tao \cite{BejenaruTao2006} is immediate. The proof of Theorem \ref{thm:LWPNLSDetailed} is complete.
\end{proof}

\begin{remark}[Persistence~of~regularity]
It is a consequence of Proposition \ref{prop:TrilinearEstimateNLS} that for $s' \geq s$, and $u_0 \in \mathcal{H}^{s'}_\Lambda$ the solution exists up to $T=T(s,\| u_0 \|_{\mathcal{H}^s_{\Lambda}})$ in $C([0,T],\mathcal{H}^{s'}_\Lambda)$. We can also extend the persistence of regularity to exponential weights like in the previous works by Papenburg \cite{Papenburg2024} and Damanik \emph{et al.} \cite{DamanikLiXu2024A,DamanikLiXu2024B}:
\begin{equation*}
\| u_0 \|^2_{\mathcal{H}^{\kappa,\omega}_\Lambda} = \sum_{n \in \Z^{\nu}} e^{2 \kappa |n|} |\hat{u}_0(\langle n \rangle_\omega)|^2.
\end{equation*}
So, we conclude that for real analytic initial data with $\| u_0 \|_{\mathcal{H}^{\kappa,\omega}_\Lambda} < \infty$ the solution persists to be real-analytic with the same exponential Fourier decay.
\end{remark}

We formulate a detailed local well-posedness result for the nonlinear Schr\"odinger equation with algebraic nonlinearity with quasi-periodic initial data in $\R^d$ with $m \geq 2$:
\begin{equation}
\label{eq:PowerNLSDetailed}
\left\{ \begin{array}{cl}
i \partial_t u + \partial_{x}^2 u &= \pm |u|^{2(m-1)} u, \quad (t,x) \in \R \times \R^d, \\
u(0) &= u_0 \in \mathcal{H}^s_\Lambda(\R^d).
\end{array} \right.
\end{equation}
Let $\Lambda$ be like in \eqref{eq:HigherDimensionalLattice} with $b_i = \nu_i-1$ and density parameter $b=b_1+\ldots+b_d$.

\smallskip

\begin{theorem}
\label{thm:LWPGenearlNLSDetailed}
Let $m \geq 3$ for $d=1$, and $m \geq 2$ for $d \geq 2$. Let $s^* = 2b \big( \frac{1}{2} - \frac{1}{2m} \big) + \big( d - \frac{d+2}{m} \big)$ and $s>s^*$. Then there is a real-analytic data-to-solution mapping for \eqref{eq:PowerNLSDetailed}
\begin{equation*}
S_T^s : \mathcal{H}^s_\Lambda(\R) \to Y_T^s \subseteq C([0,T],\mathcal{H}^s_\Lambda), \quad u_0 \mapsto u,
\end{equation*}
with $T=T(s,\| u_0 \|_{\mathcal{H}^s_\Lambda})$, which is lower semi-continuous, decreasing,  and satisfies \\ $\liminf_{a \to 0} T(s,a) \gtrsim 1$.
\end{theorem}
\begin{proof}
With the linear estimate following like above, for the proof it suffices to show the multilinear estimate with $s \geq s' > s^*$:
\begin{equation*}
\big\| \int_0^t e^{i(t-s) \partial_x^2} (u_1 \overline{u}_2 \ldots u_{2m-3} \overline{u}_{2m-2} u_{2m-1} ) dx ds \big\|_{Y_T^s} \lesssim T^\delta \| u_1 \|_{Y^s_T} \prod_{i=2}^{2m-1} \| u_i \|_{Y^{s'}_T}.
\end{equation*}
Invoking duality, we find
\begin{equation*}
\begin{split}
&\quad \big\| R_C \int_0^t e^{i(t-s) \partial_x^2} (u_1 \overline{u}_2 \ldots \overline{u}_{2m-2} u_{2m-1} ) ds \big\|_{Y^s_T} \\
&\lesssim C^s \sup_{\| v \|_{Y^0_T} = 1} \iint R_C \overline{v} R_C (u_1 \overline{u}_2 \ldots u_{2m-1} ) dx ds.
\end{split}
\end{equation*}
We omit the complex conjugation in the following to ease notation as the estimates do not depend on this. We decompose
\begin{equation*}
u_i = \sum_{C_i \in 2^{\N_0}} R_{C_i} u_i
\end{equation*}
and respecting the convolution constraint, it suffices to show multilinear estimates for
\begin{equation*}
\iint R_C v R_{C_1} u_1 \ldots R_{C_{2m-1}} u_{2m-1} dx ds.
\end{equation*}
Let $C_1^* \geq C_2^* \geq \ldots \geq C_{2m-1}^*$ denote a decreasing rearrangement. We need to consider the two cases $C_1^* \sim C$ $(I)$, and $C_1^* \sim C_2^* \gg C$ $(II)$. We turn to the estimate of $(I)$: By applying H\"older's inequality and Proposition \ref{prop:MultilinearEstimate} it follows that
\small
\begin{equation*}
\begin{split}
&\quad \big| \iint R_C v R_{C_1} u_1 \ldots R_{C_{2m-1}} u_{2m-1} dx ds \big| \\
&\leq \| R_C v R_{C_2^*} u_2^* R_{C_3^*} u_3^* \ldots R_{C_m^*} u_m^* \|_{L_T^2 \mathcal{L}^2_x} \| R_{C_1^*} u_1^* R_{C_{m+1}^*} u_{m+1}^* \ldots R_{C_{2m-1}^*} u_{2m-1}^* \|_{L_T^2 \mathcal{L}^2_x} \\
&\lesssim T^{2 \delta} (C_2^*)^{2r} (C_3^* \ldots C_m^*)^r (C_{m+1}^*)^{2r} (C_{m+2}^* \ldots C_{2m-1}^*)^r \| R_C v \|_{Y^0_T} \prod_{i=1}^{2m-1} \| R_{C_i} u_i \|_{Y_T^0}
\end{split}
\end{equation*}
\normalsize
for $r > b \big( \frac{1}{2} - \frac{1}{2m} \big) + \frac{d}{2} - \frac{d+2}{2m}$. The claim follows then from straight-forward dyadic summation.

We turn to Case $(II)$: In the following denote with $C_1^{**} \geq C_2^{**} \geq \ldots C_{2m}^{**}$. We use again H\"older's inequality to obtain
\small
\begin{equation*}
\begin{split}
&\quad \big| \iint R_C v R_{C_1} u_1 \ldots R_{2m-1} u_{2m-1} dx ds \big| \\
&\leq \| R_{C_1^*} u_1^* R_{C_3^*} u_3^* \ldots R_{C_{m+1}^*} u_{m+1}^* \|_{L_T^2 \mathcal{L}^2_x} \| R_C v R_{C_2^*} u_2^* R_{C_{m+2}^*} u_{m+2}^* \ldots R_{C_{2m-1}^*} u_{2m-1}^* \|_{L_T^2 \mathcal{L}^2_x} \\
&\lesssim T^{2\delta} (C_3^{**})^{2r} (C_4^{**})^{2r} (C_5^{**} \ldots C_{2m}^{**})^r \| R_C v \|_{Y^0_T} \prod_{i=1}^{2m-1} \| R_{C_i} u_i \|_{Y^0_T}.
\end{split}
\end{equation*}
\normalsize
With this estimate at hand, the claim is immediate from dyadic summation. The proof is complete.
\end{proof}

\section{Examples}
\label{section:Examples}
In this section we supplement the Strichartz estimates and the local well-posedness result for the cubic nonlinear Schr\"odinger equation on $\Lambda \subseteq \R^d$, $d \in \{1,2\}$, with examples, which show sharpness of the regularity up to endpoints. For simplicity, suppose that the coefficients are of unit length $\omega^{(j)}_i \in (1/2,1]$. 

\subsection{Sharpness of Strichartz estimates on finite times}

The examples are variations of a common theme: For almost-periodic functions at given height $|n| \sim C$, there are many frequencies $|\langle n \rangle_\omega | \lesssim 1$, which are consequently not strongly oscillating for finite times. We only consider the one-dimensional case $\Lambda = \omega \cdot \Z^{\nu}$. The generalization of the $\mathcal{L}^4$-estimate to the case $\Lambda \subseteq \R^2$ is straight-forward.
\begin{proposition}
Let $p \in \{4,6\}$. The estimate
\begin{equation}
\label{eq:NecessaryStrichartz}
\| e^{it \partial_x^2} R_C f \|_{L_t^p([0,1],\mathcal{L}^p_x)} \lesssim C^s \| f \|_{\mathcal{L}^2_x}
\end{equation}
fails for $s<(\nu-1) \big( \frac{1}{2}-\frac{1}{p} \big)$.
\end{proposition}
\begin{proof}
We consider
\begin{equation*}
f = \sum_{\substack{|n| \sim C, \\ |\langle n \rangle_\omega| \lesssim 1}} e^{i \langle n \rangle_\omega x}.
\end{equation*}
With $\langle n \rangle_\omega = n \cdot \omega$, $\omega_i \neq 0$, for any $n_1, \ldots, n_{\nu-1} \sim C$, we can choose $n_{\nu} \in \Z$ such that $|\langle n \rangle_\omega| \leq  1$. Note that $|n_d| \sim C$. Clearly, $\| f \|_{\mathcal{L}^2_x} \sim C^{\frac{\nu-1}{2}}$. We compute for $g = \sum_{n \in \Z^{\nu}} e^{i \langle n \rangle_\omega x} a_n$ that
\begin{equation*}
\| g \|^4_{\mathcal{L}^4_x} = \sum_{n \in \Z^{\nu}} \sum_{\substack{n_1,n_3, \\ n= n_1+n_2=n_3+n_4}} a_{n_1} a_{n_2} \overline{a}_{n_3} \overline{a}_{n_4}.
\end{equation*}
Consequently,
\begin{equation*}
\| f \|^4_{\mathcal{L}^4_x} \sim C^{3(\nu-1)}.
\end{equation*}
This shows that \eqref{eq:NecessaryStrichartz} for $p=4$
can only be valid for $s \geq \frac{\nu-1}{4}$. This establishes the claim for $p=4$.

\smallskip

Moreover,
\begin{equation*}
\| g \|^6_{\mathcal{L}^6_x} = \sum_{n \in \Z^{\nu}} \sum_{\substack{n_1,n_2,n_4,n_5, \\ n= n_1+n_2+n_3=n_4+n_5+n_6}} a_{n_1} a_{n_2} a_{n_3} \overline{a}_{n_4} \overline{a}_{n_5} \overline{a}_{n_6}.
\end{equation*}
For this reason,
\begin{equation*}
\| f \|_{\mathcal{L}^6_x}^6 \sim C^{5(\nu-1)}.
\end{equation*}
This shows that \eqref{eq:NecessaryStrichartz} for $p=6$ can only be true for $s \geq \frac{\nu-1}{3}$. The proof is complete.
\end{proof}

\subsection{Ill-posedness of the nonlinear Schr\"odinger equation}

We consider the cubic nonlinear Schr\"odinger equation on $\Lambda \subseteq \R^d$, $d \in \{1,2\}$. We let $\Lambda = \omega^{(1)} \cdot \Z^{\nu_1} \times \omega^{(2)} \cdot \Z^{\nu_2}$ and $b=(\nu_1-1)+(\nu_2-1)$.
\begin{proposition}
The first Picard iterate 
\begin{equation}
\label{eq:FirstPicardIterate}
\big\| \int_0^t e^{i(t-s) \Delta} \big( |e^{is \Delta} f|^2 e^{is \Delta} f \big) ds \big\|_{C_T \mathcal{H}^s_\Lambda} \lesssim \| f \|^3_{\mathcal{H}^s_\Lambda}.
\end{equation}
is unbounded for $s< \frac{b}{2}$.
\end{proposition}
\begin{proof}
The choice of data is the same as in the previous section. We have
\begin{equation*}
\big( \big| e^{is \Delta} f \big|^2 e^{is \Delta} f \big) \widehat (n) = \sum_{n=n_1-n_2+n_3} e^{is (\langle n_1 \rangle_\omega^2 - \langle n_2 \rangle_\omega^2 + \langle n_3 \rangle_\omega^2)},
\end{equation*}
which gives
\begin{equation*}
\begin{split}
&\quad \int_0^t e^{i(t-s)\Delta} \big| e^{is \Delta} f \big|^2 e^{is \Delta} f ds \\ &= \sum_n e^{i t \langle n \rangle_\omega^2} e^{i x \langle n \rangle_\omega} \underbrace{\sum_{n=n_1-n_2+n_3} \int_0^t e^{is(\langle n_1 \rangle_\omega^2 - \langle n_2 \rangle^2_\omega + \langle n_3 \rangle_\omega^2 - \langle n \rangle_\omega^2)} ds}_{c(t,n)}.
\end{split}
\end{equation*}
Let $\Omega(n_1,n_2,n_3) = \langle n_1 \rangle_\omega^2 - \langle n_2 \rangle^2_\omega + \langle n_3 \rangle_\omega^2 - \langle n_1-n_2+n_3 \rangle_\omega^2$.
We compute the integral to be
\begin{equation*}
\int_0^t e^{is \Omega(n_1,n_2,n_3)} ds = 
\begin{cases}
\frac{e^{it \Omega} -1}{i \Omega}, &\quad \Omega \neq 0, \\
t, &\quad \Omega = 0.
\end{cases}
\end{equation*}
Since $|\Omega| \lesssim 1$, we have for $0 \leq t \leq c \ll 1$ the uniform estimate
\begin{equation*}
|c(t,n)| \gtrsim t \sum_{n=n_1-n_2+n_3} 1.
\end{equation*}
We have $|n_i| \sim C$. So there are $\gtrsim C^{b}$ output frequencies of height $C$ with coefficients satisfying the bound $|c(t,n)| \gtrsim t \cdot C^{2b}$. This gives
\begin{equation*}
\big\| \int_0^t e^{i(t-s) \Delta} \big( |e^{is \Delta} f|^2 e^{is \Delta} f \big) ds \big\|_{C_T \mathcal{H}^s_\Lambda} \gtrsim T C^s C^{2b} C^{\frac{b}{2}}.
\end{equation*}
Secondly, $\| f \|_{\mathcal{H}^s_\Lambda} \sim C^{\frac{b}{2}+s}$. Consequently, \eqref{eq:FirstPicardIterate} can only be true for $s \geq \frac{b}{2}$.
\end{proof}

We conclude with two remarks:
\begin{remark}
Note that it is actually enough to show sharpness of the local well-posedness result to infer sharpness of the $L^4$-Strichartz estimate. Indeed, as soon as there is an improved Strichartz estimate, the improvement of the local well-posedness theory follows readily following along the above arguments.  

Secondly, comparing with the periodic case, it is tempting to conjecture that the local well-posedness result is also sharp in the general quasi-periodic case $\Lambda \subseteq \R^d$ for $d \geq 3$. It remains unclear how to corroborate this by examples though.
\end{remark}

\begin{remark}[Quintic~NLS]
We remark that the same computation shows that the quintic nonlinear Schr\"odinger equation posed on $\Lambda = \omega \cdot \Z^{\nu}$
\begin{equation*}
\left\{ \begin{array}{cl}
i \partial_t u+ \partial_{x}^2 u &= |u|^4 u, \quad (t,x) \in \R \times \R, \\
u(0) &= u_0 \in \mathcal{H}^s_\Lambda
\end{array} \right.
\end{equation*}
can only be $C^3$-well-posed for $s > \frac{1}{2}$. The well-posedness result we showed for $s>\frac{2}{3}$ is actually not sharp: If the frequency distribution is given by
\begin{equation*}
A = |\iint R_{M_1} u_1 R_{M_2} u_2 R_{M_3} u_3 R_{M_4} u_4 R_{M_5} u_5 R_{M_6} u_6 dx dt|
\end{equation*}
with $M_1 \sim M_2 \gg M_3 \sim M_4 \gg M_5 \sim M_6$ the application of two trilinear estimates gives
\begin{equation*}
A \lesssim M_3^{2r} M_4^{2r} M_5^r M_6^r \prod_{i=1}^6 \| R_{M_i} u_i \|_{Y_T^0}.
\end{equation*}
for $r>\frac{1}{3}$. We have the alternative estimate from two bilinear $L^2$-Strichartz estimates:
\begin{equation*}
A \lesssim M_3^{2r} M_4^{2r} M_5 M_6 \prod_{i=1}^6 \| R_{M_i} u_i \|_{Y_T^0}.
\end{equation*}
So, choosing a threshold $M_5 \lesssim M_3^\alpha$ will allow us to improve the local well-posedness. But since we cannot reach $s>\frac{1}{2}$ as suggested by the example, we omit the details.
\end{remark}

\section*{Appendix: Quasiperiodic solutions to the KdV equation}

We remark that the above arguments to show local well-posedness for the nonlinear Schr\"odinger equation do not take advantage of the resonance.
In this section we consider the KdV with derivative nonlinearity:
\begin{equation}
\label{eq:KdVAppendix}
\left\{ \begin{array}{cl}
\partial_t u + \partial_x^3 u &= u \partial_x u, \quad (t,x) \in \R \times \R, \\
u(0) &= u_0 \in \mathcal{H}^s_\Lambda.
\end{array} \right.
\end{equation}
The systematic study of Strichartz estimates is the first cornerstone to show local well-posedness. The second ingredient already used by Bourgain \cite{Bourgain1993B} is to use large resonance to ameliorate the derivative loss. The resonance function quantifies to what extend free solutions can form a free solution again through the nonlinear interaction and reads for the KdV equation:
\begin{equation*}
\Omega(\xi_1,\xi_2) = (\xi_1+\xi_2)^3 - \xi_1^3 - \xi_2^3 = 3 (\xi_1 + \xi_2) \xi_1 \xi_2.
\end{equation*}
The Fourier restriction norm, or the refined adapted spaces, can recover the square root of the resonance function. Consider a frequency interaction with output frequency $N \in 2^{\Z}$ and input frequencies $N_1$, $N_2 \in 2^{\Z}$. Let $N_{\max} = \max(N,N_1,N_2)$, and $N_{\min} = \min(N,N_1,N_2)$. The resonance relation recovered above shows that for $|\xi_i| \sim N_i$ and $|\xi_1+\xi_2| \sim N$ we have 
\begin{equation*}
|\Omega(\xi_1,\xi_2)| \sim N_{\max}^2 N_{\min}.
\end{equation*}
The purpose of this section is devoted to explain how Tsugawa's result is a consequence of the Strichartz estimates at hand and the resonance analysis due to Bourgain \cite{Bourgain1993B}. This will help for future analysis of more involved models.

\smallskip

We also mention the recent result of Papenburg \cite{Papenburg2024} who showed unconditional local well-posedness for exponentially decaying Fourier coefficients for more general dispersive equations with derivative nonlinearity. However, the solutions were only verified to satisfy a weaker exponential Fourier decay.

\subsection{Local well-posedness in the periodic case}

The resonance relation highlights that interactions with low frequencies are problematic.
By subtracting the mean-value and derivative nonlinearity, we can suppose that the mean is identically vanishing for all times: $\hat{u}(t,0) = 0$. With the $L^4$-Strichartz estimates on the torus
\begin{equation*}
\| e^{t \partial_x^3} f \|_{L_t^4([0,T],L^4(\T))} \lesssim T^{\frac{1}{8}} \| f \|_{L^2(\T)}
\end{equation*}
it is nowadays standard to close the contraction mapping argument. 

 On the torus we define the function spaces\footnote{Here we use a Besov variant in the frequencies to simplify summability.}:
\begin{equation*}
\| u \|_{Y^s_T} = \sum_{n \in \Z \backslash 0} \langle n \rangle^s \| e^{it \langle n \rangle_\omega^3} \hat{u}(t,n) \|_{V^2_T}.
\end{equation*}
For the nonlinear estimate, it suffices to prove after invoking duality for $N_2 \lesssim N_1 \sim N$:
\begin{equation}
\label{eq:TrilinearEstimateTorus}
\big| \iint_{[0,T] \times \T} P_N v P_{N_1} u_1 P_{N_2} u_2 dx dt \big| \lesssim T^\delta N^{-1} \| P_N v \|_{Y^0_T} \| P_{N_1} u_1 \|_{Y^0_T} N_2^s \| P_{N_2} u_2 \|_{Y^0_T}.
\end{equation}
The high frequency gain will compensate the derivative loss.

To show \eqref{eq:TrilinearEstimateTorus}, we use the resonance relation. For one function, say $v$, we can suppose that the modulation is larger than the resonance:\footnote{We gloss over the technicality of considering finite times. Roughly speaking, for finite times modulations can be assumed to have a minimum size.}
\begin{equation*}
(\xi,\tau) \in \text{supp}(\hat{v}(\xi,\tau)) \Rightarrow | \tau - \xi^3 | \gtrsim N^2 N_2.
\end{equation*}

By H\"older's inequality and the $L^4$-Strichartz estimate we find
\begin{equation*}
\begin{split}
\iint_{[0,T] \times \T} P_{N} v P_{N_1} u_1 P_{N_2} u_2 dx dt &\leq \| P_N v \|_{L^2_{t,x}} \|P_{N_1} u_1 \|_{L^4_{t,x}} \| P_{N_2} u_2 \|_{L^4_{t,x}} \\
&\lesssim T^{\frac{1}{4}} (N^2 N_2)^{-\frac{1}{2}} \| P_N v \|_{Y^0_T} \prod_{i=1}^2 \| P_{N_i} u_i \|_{Y^0_T},
\end{split}
\end{equation*}
which is \eqref{eq:TrilinearEstimateTorus} for $s = -\frac{1}{2}$.

\subsection{Local well-posedness in the quasiperiodic case}

Next, we extend the argument to \eqref{eq:KdVAppendix} on $\Lambda = \omega \cdot \Z^{\nu}$, $\omega \in \R^{\nu}_{>0}$ non-resonant. Let $s_1,s_2 \in \R$. We define the norm to estimate solutions by
\begin{equation*}
\| u \|_{Y_T^{s_1,s_2}} = \sum_{n \in \Z^{\nu} \backslash 0} \langle n \rangle_\omega^{s_1} \langle n \rangle^{s_2} \| e^{it \langle n \rangle_\omega^3} \hat{u}(t,n) \|_{V^2_T}.
\end{equation*}
As already observed by Tsugawa, there is a significant difference with low frequencies compared to the periodic case. Removing the zero frequency in the periodic case is enough to ensure that the frequencies have at least unit modulus. This is clearly no longer true in the quasiperiodic case and suggests to penalize low frequencies by homogeneous Sobolev norms with negative regularity. Note that in the decaying case on the real line the very low frequencies can be estimated by their small measure and dispersive effects, both absent in the quasiperiodic case. 

The generalized trilinear estimate in the quasiperiodic case reads 
\begin{equation}
\label{eq:TrilinearEstimateQuasiperiodicKdV}
\begin{split}
&\quad \lim_{L \to \infty} \frac{1}{2L} \iint_{[0,T] \times [-L,L]} P_N R_C v P_{N_1} R_{C_1} u_1 P_{N_2} R_{C_2} u_2 dx dt \\
&\lesssim C_{\min}^{s_2} N_{\max}^{-1} N_{\min}^{s_1} \|P_N R_C v \|_{Y^0_T} \prod_{i=1}^2 \| P_{N_i} R_{C_i} u_i \|_{Y^0_T}.
\end{split}
\end{equation}
To effectively lower the height, we can carry out an almost orthogonal decomposition to localize also the maximal height to cubes of size $C_{\min}$ (compare Proposition \ref{prop:BilinearStrichartzEstimateSEQ}). With the following variant of the $L^4$-estimate
from \eqref{eq:FixedTimeL4StrichartzAiryIntro}, which is a consequence of the transfer principle, we find
\begin{equation*}
\| P_N Q_C^a u \|_{L_T^4 \mathcal{L}^4_x} \lesssim T^{\frac{1}{8}} C^{\frac{b}{4}} \| u \|_{Y^0_T}.
\end{equation*}
Then \eqref{eq:TrilinearEstimateQuasiperiodicKdV} follows like above after almost orthogonal decomposition into cubes of low height, applying H\"older's inequality and the $L^4$-Strichartz estimate recorded in the above display. This essentially recovers Tsugawa's result \cite[Theorem~1.1]{Tsugawa2012}. The $L^4$-estimate was implicitly proved in \cite[Lemma~3.1]{Tsugawa2012}. 

\subsection{Summary}

It is well-known since the works of Bourgain \cite{Bourgain1993A,Bourgain1993B} how to combine the resonance relation with Strichartz estimates. The resonance relation remains the same, and it remains to find substitutes for the Strichartz estimates in the quasi-periodic case. Above we have demonstrated how Strichartz estimates depending on the height follow from decoupling and counting arguments. We hope to apply this approach to further models.

\section*{Acknowledgements}

Financial support from the Humboldt foundation (Feodor-Lynen fellowship) and partial support by the NSF grant DMS-2054975 is gratefully acknowledged. Some of this work was carried out at Bielefeld University in June 2024, and I would like to thank Sebastian Herr and his working group for their kind hospitality.

\bibliographystyle{plain}

\end{document}